\documentclass{amsart}
\usepackage{graphicx}
\usepackage{amssymb}
\usepackage{wrapfig}
\usepackage[bookmarksnumbered,colorlinks,plainpages,backref]{hyperref}

\newtheorem{thm}{Theorem}[section]
\newtheorem{cor}[thm]{Corollary}
\newtheorem{lem}[thm]{Lemma}
\newtheorem{prop}[thm]{Proposition}
\theoremstyle{definition}

\theoremstyle{remark}
\newtheorem{rem}[thm]{Remark}
\newtheorem{question}[thm]{Question}
\numberwithin{equation}{section}
\newcommand{\norm}[1]{\left\Vert#1\right\Vert}
\newcommand{\abs}[1]{\left\vert#1\right\vert}
\newcommand{\set}[1]{\left\{#1\right\}}

\newcommand{\dbar}{\bar\partial}
\newcommand{\ddbar}{\partial\bar\partial}

\DeclareMathOperator{\dom}{Dom}

\DeclareMathOperator{\re}{Re}
\DeclareMathOperator{\im}{Im}

\DeclareMathOperator{\dist}{dist}

\DeclareMathOperator{\Div}{div}
\DeclareMathOperator{\Ric}{Ric}
\begin{document}

\title[Sobolev Regularity for the Bergman Projection in Hermitian manifolds]{Sobolev Regularity for the Bergman Projection on Relatively Compact Domains in Hermitian manifolds}%
\author{Phillip S. Harrington}%
\address{SCEN 309, 1 University of Arkansas, Fayetteville, AR 72701}%
\email{psharrin@uark.edu}%

\subjclass[2010]{32U10, 32T35, 32Q28}
\keywords{Diederich-Forn{\ae}ss Index, Bergman Projection, Hermitian manifolds}%

\begin{abstract}
  Generalizing a result of Berndtsson and Charpentier, we provide sufficient conditions for $L^2$ Sobolev regularity of the Bergman projection acting on $L^2$ sections of a holomorphic line bundle restricted to a relatively compact domain with Lipschitz boundary in a Hermitian manifold.  We provide examples to show that our methods work for domains in Hopf manifolds endowed with a suitable Hermitian metric.
\end{abstract}
\maketitle



\section{Introduction}

Let $M$ be a complex manifold of dimension $n\geq 2$, and equip the holomorphic tangent bundle $T^{1,0}(M)$ with a Hermitian metric.  Let $\Omega\subset M$ be a relatively compact domain with Lipschitz boundary.  Let $L\rightarrow M$ be a holomorphic line bundle equipped with a Hermitian metric.  Let $L^2(\Omega,L)$ denote the space of $L^2$ holomorphic sections from $\Omega$ to $L$, and let $P$ denote the Bergman projection on $L^2(\Omega,L)$, i.e., the orthogonal projection onto the subspace of $L^2$ holomorphic sections of $L$ over $\Omega$ (see \cite{Szo22} or \cite{GGV22} for recent work on the Bergman space in this setting).  Our primary question in this paper is the following:
\begin{question}
\label{q:key_question}
  Given $s>0$, what is a sufficient condition on $\Omega$ and the Hermitian metrics on $M$ and $L$ for the Bergman projection to be continuous in the $L^2$ Sobolev space $W^s(\Omega,L)$?
\end{question}
Our long term goal is to understand when Question \ref{q:key_question} has a positive answer for all $s>0$ on a domain $\Omega$ with smooth boundary.  For the present paper, we follow Berndtsson and Charpentier \cite{BeCh00} in considering regularity results for smaller values of $s$ when the boundary is merely Lipschitz.  To aid in the exposition, we will focus on the special case in which $L$ is the trivial line bundle with a Hermitian metric represented by a weight function $\psi$.  More precisely, for $\psi\in C^2(M)$, we define the weighted $L^2$ inner product by
\[
  \left<f,g\right>_{L^2(\Omega,\psi)}=\int_\Omega\left<f,g\right>e^{-\psi}dV,
\]
with associated norm $\norm{f}_{L^2(\Omega,\psi)}$.  Note that $\norm{f}_{L^2(\Omega,\psi)}$ is comparable to $\norm{f}_{L^2(\Omega)}$ under our hypotheses.  Let $P_\psi$ denote the orthogonal projection onto the space of $L^2$ holomorphic functions with respect to the weighted inner product, i.e., the Bergman projection weighted by $\psi$.  In this setting, our question becomes:
\begin{question}
\label{q:key_question_simplified}
  Given $s>0$, what is a sufficient condition on $\Omega$, the Hermitian metric on $M$, and $\psi\in C^2(M)$ for the weighted Bergman projection $P_\psi$ to be continuous in the $L^2$ Sobolev space $W^s(\Omega)$?
\end{question}

Our starting point is Kohn's formula relating the Bergman projection to the $L^2$-minimal solution to the inhomogeneous Cauchy-Riemann equations.  If $S:L^2_{0,1}(\Omega)\cap\ker\dbar\rightarrow L^2(\Omega)$ is the $L^2(\Omega)$-minimal solution operator to the inhomogenous Cauchy-Riemann equations (also known as Kohn's solution), then the Bergman projection satisfies $P=I-S\dbar$.  Kohn's groundbreaking work in \cite{Koh63,Koh64} established the existence and regularity of $S$ when the boundary of $\Omega$ is strictly pseudoconvex.  Because of this, it has long been known that the Bergman projection is regular in $W^s(\Omega)$ for all $s\in\mathbb{R}$ when the boundary of $\Omega$ is strictly pseudoconvex.  We note that Kohn and many of the authors cited below focus on the $\dbar$-Neumann problem.  As shown by Boas and Straube in \cite{BoSt90}, regularity for the Bergman projection and regularity for the $\dbar$-Neumann problem (in the appropriate degrees) are equivalent for smoothly bounded domains in $\mathbb{C}^n$.  This equivalence may not hold in general complex manifolds, and we will see examples of domains in which we can prove regularity for the Bergman projection without regularity for the $\dbar$-Neumann problem.  In particular, we will see that we do not need a solution operator $S$ for all inhomogeneous Cauchy-Riemann equations, but only for the special equation given in \eqref{eq:solution_identity} below.  This will allow us to estimate the Bergman projection on domains which are not Stein.

When the boundary of $\Omega$ is not strictly pseudoconvex, more information is needed.  Indeed, Barrett proved in \cite{Bar92} that for any $s>0$, there exists a smoothly bounded pseudoconvex domain in $\mathbb{C}^2$ on which the Bergman projection is not continuous in $W^s(\Omega)$.  These are the famous worm domains of Diederich and Forn{\ae}ss \cite{DiFo77a}.  When $M=\mathbb{C}^n$, many sufficient conditions for regularity of the Bergman Projection have been studied.  In \cite{BoSt91}, Boas and Straube proved that the Bergman projection is continuous in $W^s(\Omega)$ for all $s\in\mathbb{R}$ whenever $\Omega\subset\mathbb{C}^n$ admits a smooth defining function that is strictly plurisubharmonic on the boundary of $\Omega$.  Building on this idea, Kohn found a way in \cite{Koh99} to quantify regularity of the Bergman projection using the Diederich-Forn{\ae}ss index.  In \cite{DiFo77b}, Diederich and Forn{\ae}ss proved that for every bounded pseudoconvex domain in $\mathbb{C}^n$ with $C^2$ boundary, there exists a number $0<\eta<1$ and a defining function $\rho_\eta$ for $\Omega$ such that $-(-\rho_\eta)^\eta$ is strictly plurisubharmonic on $\Omega$.  The Diederich-Forn{\ae}ss index, which we will denote $DF(\Omega)$, is the supremum of all such $\eta$ for a given domain $\Omega$.  Kohn found a lower bound for the value of $s$ for which the Bergman projection is continuous in $W^s(\Omega)$ using $DF(\Omega)$ and certain regularity properties of the defining functions $\rho_\eta$ \cite{Koh99}.  Other authors have sharpened Kohn's results in \cite{Har11}, \cite{PiZa14}, \cite{Liu22}, and \cite{LiSt22}.  We highlight the work of Liu and Straube in \cite{LiSt22}, in which they prove that the Bergman projection is continuous in $W^s(\Omega)$ for all $s\in\mathbb{R}$ whenever $\Omega\subset\mathbb{C}^2$ is a smoothly bounded pseudoconvex domain and $DF(\Omega)=1$.  Note that this is a generalization of Boas and Straube's result in \cite{BoSt91} when $n=2$ by a result of Gallagher and Forn{\ae}ss \cite{HeFo07}.

On domains in $\mathbb{C}^n$ for which the boundary is only Lipschitz, Berndtsson and Charpentier proved that the Bergman projection is regular in $W^s(\Omega)$ whenever $0<s<\frac{1}{2}DF(\Omega)$.  Since the author proved that $DF(\Omega)>0$ whenever $\Omega\subset\mathbb{C}^n$ is a bounded pseudoconvex domain with Lipschitz boundary \cite{Har08a}, the Berndtsson and Charpentier result can be used to prove that the Bergman projection always admits some regularity in Sobolev spaces on bounded pseudoconvex domains in $\mathbb{C}^n$ with Lipschitz boundaries.  Our goal in the present paper is to see what the methods of Berndtsson and Charpentier reveal about pseudoconvex domains in more general complex manifolds.  We note that Kohn's original work in \cite{Koh63,Koh64} applies to strictly pseudoconvex domains in complex manifolds, without the assumption of a K\"ahler metric.  Clearly some hypotheses about the curvature of the ambient manifold must be introduced, but we will see that these methods work on domains which are not Stein, so we can obtain regularity for the Bergman projection on domains for which the inhomogeneous Cauchy-Riemann equations are ill-posed.

The curvature terms arising in the $L^2$ theory for a K\"ahler manifold are well-known from the celebrated Bochner-Kodaira-Nakano identity.  On non-K\"ahler manifolds, the necessary generalizations were initially computed by Griffiths in \cite{Gri66} and developed by Demailly in \cite{Dem86} (see \cite{Dem12} for a detailed exposition).  One of our goals for the present paper is to demonstrate that in the crucial bottom-degree case, the necessary curvature terms are still quite easy to compute and relate to known quantities.  First, we will need the complex Hessian of the weight function $\psi$, which should be thought of as the curvature of the trivial line bundle under the metric induced by $\psi$ (see \eqref{eq:line_bundle_curvature_computation} below).  Next, we need the torsion of the Chern connection for $T^{1,0}(M)$, which can be easily computed using the exterior derivative of the K\"ahler form (see \eqref{eq:Kahler_differential} below).  Finally, we need a curvature term which can be computed with the same formula used to compute the Ricci tensor of a K\"ahler manifold (see \eqref{eq:ricci_computation} below), and which represents a scalar multiple of the first Chern class of $T^{1,0}(M)$.  To consider the Bergman projection in the space $L^2_{p,q}(\Omega,\psi)$ for $q\geq 1$ or $p\geq 1$ with a non-K\"ahler metric, the curvature terms arising in the $L^2$ estimates would be significantly more complicated, and something closer to the methods of \cite{Gri66} or \cite{Dem86} would be necessary (although there are also significant simplifications when $p=n$; see (3.2) in \cite{Dem86}).  For example, compare Th\'{e}or\`{e}me 2.12 and (3.1) in \cite{Dem86} or Theorem VII-1.4 and VII-(2.1) in \cite{Dem12}, in which we see that the torsion operator $T_\omega$ arising in the $L^2$ estimates depends on $\ddbar\omega$ and commutation relations involving a torsion operator and its adjoint.

To state our results precisely, we need to fix some notation.  We denote the Hermitian metric on $T^{1,0}(M)$ by $\left<\cdot,\cdot\right>$, and the K\"ahler form by $\omega$, i.e., $\omega(Z,\bar W)=\frac{i}{2}\left<Z,\bar W\right>$ for all $Z,W\in T^{1,0}(M)$.  Let $\nabla$ denote the Chern connection on $T^{1,0}(M)$ with respect to the Hermitian metric.  We also use $\nabla$ to denote the unique affine connection on $T(M)$ induced by the Chern connection.  Since $M$ is not necessarily K\"ahler, we define the torsion tensor
\[
  T_\nabla(X,Y)=\nabla_X Y-\nabla_Y X-[X,Y]
\]
for all $X,Y\in T(M)$.  We also define the usual curvature tensor
\[
  R_\nabla(X,Y)=\nabla_X\nabla_Y-\nabla_Y\nabla_X-\nabla_{[X,Y]}
\]
for all $X,Y\in T(M)$.  We extend these tensors to the complexified tangent bundle so that they are $\mathbb{C}$-linear in $X$ and $Y$.  For any $U\subset M$, if $\{W_j\}_{1\leq j\leq n}$ is a basis for $T^{1,0}(U)$ and $\{\theta^j\}_{1\leq j\leq n}$ is the dual basis for $\Lambda^{1,0}(U)$, then we define our key curvature term by
\[
  \Theta_M(Z_1,\bar Z_2):=\sum_{j=1}^n\theta^j\left(R_{\nabla}(Z_1,\bar Z_2)W_j\right).
\]
for any $Z_1,Z_2\in T^{1,0}(U)$.  This definition is independent of the choice of basis for $T^{1,0}(U)$.  When $M$ is K\"ahler, it is not difficult to check that this agrees with the Ricci tensor $\Ric_\nabla$ using the first Bianchi identity for K\"ahler manifolds.  Note that $\Theta_M$ is the trace of the curvature form, so it represents a scalar multiple of the first Chern class for $T^{1,0}(M)$.

To handle the torsion, we define an operator $\tau_\nabla:\Lambda^q(M)\rightarrow\Lambda^{q+1}(M)$ by setting
\begin{equation}
\label{eq:tau_defn}
  (\tau_\nabla\phi)(X,Y)=\phi(T_\nabla(X,Y))
\end{equation}
for all $\phi\in\Lambda^1(M)$ and $X,Y\in T(M)$, and extending $\tau_\nabla$ as a derivation to $\Lambda^q(M)$ for $q>1$.  Since $T_\nabla$ is skew-symmetric, $\tau_\nabla\phi$ is a well-defined $2$-form.  Because $\nabla$ is the Chern connection, we will see that $\tau_\nabla$ maps $\Lambda^{1,0}(M)$ to $\Lambda^{2,0}(M)$ (this is a consequence of \eqref{eq:torsion_vanishing} below).  We will often compose $\tau_\nabla$ with the musical isomorphism: for any $Z\in T^{1,0}(M)$, we define $Z^\flat\in\Lambda^{1,0}(M)$ by $Z^\flat(W)=\left<W,Z\right>$ for all $W\in T^{1,0}(M)$, and hence
\[
  (\tau_\nabla Z^\flat)(W_1,W_2)=\left<T_\nabla(W_1,W_2),Z\right>\text{ for all }Z,W_1,W_2\in T^{1,0}(M).
\]

Now we may state our key result:
\begin{thm}
\label{thm:DF_psh}
  Let $M$ be a Hermitian manifold of dimension $n\geq 2$ and let $\Omega\subset M$ be a relatively compact domain with Lipschitz boundary.  Let $\psi\in C^2(M)$.  Suppose that for some $0\leq a<b\leq 1$, there exists a Lipschitz defining function $\rho$ for $\Omega$ such that $\rho|_\Omega\in C^2(\Omega)$ and for every $Z\in T^{1,0}(\Omega)$
  \begin{equation}
  \label{eq:DF_psh}
    \left((-\rho)^\eta\left(\Theta_M+\ddbar\psi\right)-\ddbar(-\rho)^\eta\right)(Z,\bar Z)\geq(-\rho)^\eta\abs{\tau_\nabla Z^\flat}^2
  \end{equation}
  on $\Omega$ for $\eta=a$ and $\eta=b$.  Then the Bergman projection $P_\psi$ is continuous in $W^s(\Omega)$ for all $0\leq s<\frac{b}{2}$.
\end{thm}
We note that the value of $a$ does not play a role in the conclusion of Theorem \ref{thm:DF_psh}, and in many of our examples we will take $a=0$.  Its relevance for Theorem \ref{thm:DF_psh} lies in requiring that \eqref{eq:DF_psh} holds for $\eta$ in some open interval, and not only at an isolated point.

When the boundary of $\Omega$ is also $C^2$, we may simply state that $\rho$ is a $C^2$ defining function for $\Omega$.  However, we will see in Section \ref{sec:C2_example} that we can construct domains with Lipschitz boundaries satisfying the hypotheses of Theorem \ref{thm:DF_psh}.  We note that it is dangerous to assume that a Lipschitz defining function can always be regularized to obtain a function which is $C^2$ in the interior without sacrificing \eqref{eq:DF_psh}.  Our example in Section \ref{sec:Example} is foliated by compact complex submanifolds, so any plurisubharmonic function on this domain must be constant on each such submanifold.  This means that any regularizing procedure must take this global information into account, so merely integrating against a locally defined mollifier in order to regularize is perilous.

If $M=\mathbb{C}^n$ with the Euclidean metric and $\psi\equiv 0$, then $\tau_\nabla\equiv 0$ and $\Theta_M\equiv 0$, so \eqref{eq:DF_psh} holds when $\eta=0$.  If $-(-\rho)^b$ is plurisubharmonic for some $0<b\leq 1$, then \eqref{eq:DF_psh} will also hold at $\eta=b$.  Hence, we may recover the result of \cite{BeCh00}.  If $M=\mathbb{CP}^n$ with the Fubini-Study metric and $\psi\equiv 0$, then $\Theta_M=\Ric_{\nabla}$ is positive definite, so by the same reasoning we may recover the result of \cite{CSW04}.  When $M$ is Stein, we may set $\psi$ equal to a multiple of a smooth, strictly plurisubharmonic exhaustion function for $M$ to construct weighted Bergman projections with improved regularity, as observed by Kohn in \cite{Koh73}.

We will see in Section \ref{sec:hopf_manifolds} that a Hopf manifold can be constructed with a Hermitian metric in which \eqref{eq:DF_psh} holds when $\eta=0$, thus generalizing the result of Berndtsson and Charpentier to domains in a non-K\"ahler manifold, provided that $DF(\Omega)>0$ on those domains.  On the other hand, we will show in Proposition \ref{prop:Kahler_optimal} that on any K\"ahler manifold, optimal values for $\eta$ in \eqref{eq:DF_psh} are obtained with a K\"ahler metric and a suitable weight function $\psi$.  Hence, there are no advantages to using a non-K\"ahler metric on a K\"ahler manifold, but there are Hermitian manifolds which do not admit a K\"ahler metric on which the hypotheses of Theorem \ref{thm:DF_psh} are satisfied.

We can easily generalize our previous results to the context of Question \ref{q:key_question}.  Let $L\rightarrow M$ be a holomorphic line bundle over $M$ with a metric denoted by $\left<\cdot,\cdot\right>_L$.  Let $\nabla^L$ denote the Chern connection on $L$.  For a non-trivial section $\mathcal{S}$ of $L$, we define
\[
  \Theta_L(Z_1,\bar Z_2):=\abs{\mathcal{S}}^{-2}_L\left<R_{\nabla^L}(Z_1,\bar Z_2)\mathcal{S},\mathcal{S}\right>_L
\]
for any $Z_1,Z_2\in T^{1,0}(M)$.  This is clearly independent of the choice of section.  Theorem \ref{thm:DF_psh} generalizes to:
\begin{thm}
\label{thm:DF_psh_line_bundle}
  Let $M$ be a Hermitian manifold of dimension $n\geq 2$ and let $\Omega\subset M$ be a relatively compact domain with Lipschitz boundary.  Let $L\rightarrow M$ be a holomorphic line bundle equipped with a $C^2$ Hermitian metric.  Suppose that for some $0\leq a<b\leq 1$, there exists a Lipschitz defining function $\rho$ for $\Omega$ such that $\rho|_\Omega\in C^2(\Omega)$ and for every $Z\in T^{1,0}(\Omega)$
  \begin{equation}
  \label{eq:DF_psh_line_bundle}
    \left((-\rho)^\eta\left(\Theta_M+\Theta_L\right)-\ddbar(-\rho)^\eta\right)(Z,\bar Z)\geq(-\rho)^\eta\abs{\tau_\nabla Z^\flat}^2
  \end{equation}
  on $\Omega$ for $\eta=a$ and $\eta=b$.  Then the Bergman projection is continuous in $W^s(\Omega,L)$ for all $0\leq s<\frac{b}{2}$.
\end{thm}
The proof of Theorem \ref{thm:DF_psh_line_bundle} is a trivial modification of the proofs of Theorem \ref{thm:DF_psh}, so we omit the details.  The key observation is that we may treat the local computations in the proof of Proposition \ref{prop:twisted_BKMKH} as computations on a local trivialization of an arbitrary holomorphic line bundle.  Furthermore, \eqref{eq:line_bundle_curvature_computation} below justifies replacing $\ddbar\psi$ with $\Theta_L$.

As noted in the discussion of Theorem \ref{thm:DF_psh}, it is dangerous to assume that a Lipschitz function satisfying \eqref{eq:DF_psh} can be regularized.  This, for example, is why the standard regularizing result of Richberg \cite{Ric68} is stated for strictly plurisubharmonic functions.  It is not clear that Richberg's methods allow us to simultaneously regularize both $-(-\rho)^a$ and $-(-\rho)^b$ so that \eqref{eq:DF_psh} holds for both $\eta=a$ and $\eta=b$, so we will assume that $a=0$ in this case.  If we adapt Richberg's techniques, we obtain the following:
\begin{thm}
\label{thm:DF_strict_psh_line_bundle}
  Let $M$ be a Hermitian manifold of dimension $n\geq 2$ and let $\Omega\subset M$ be a relatively compact domain with Lipschitz boundary.  Let $L\rightarrow M$ be a holomorphic line bundle equipped with a $C^2$ Hermitian metric satisfying
  \begin{equation}
  \label{eq:positivity_condition_line_bundle}
    \left(\Theta_M+\Theta_L\right)(Z,\bar Z)\geq\abs{\tau_\nabla Z^\flat}^2
  \end{equation}
  on $\Omega$ for all $Z\in T^{1,0}(\Omega)$.  Suppose that for some $0<\eta\leq 1$, there exists a continuous function $\rho$ that is comparable to $-\dist(\cdot,\partial\Omega)$ on $\Omega$ such that for every relatively compact set $U\subset\Omega$ there exists a constant $A_U>0$ satisfying
  \begin{equation}
  \label{eq:DF_strict_psh_line_bundle}
    \left((-\rho)^\eta\left(\Theta_M+\Theta_L\right)-\ddbar(-\rho)^\eta\right)(Z,\bar Z)\geq(-\rho)^\eta\abs{\tau_\nabla Z^\flat}^2+A_U\abs{Z}^2
  \end{equation}
  in the distribution sense on $U$ for every $Z\in T^{1,0}(U)$.  Then the Bergman projection is continuous in $W^s(\Omega,L)$ for all $0\leq s<\frac{\eta}{2}$.
\end{thm}
We omit the proof, as it is a straightforward application of Richberg's techniques in \cite{Ric68}.

We conclude with an open problem that we hope will motivate future research: if we assume that the boundary of $\Omega$ is smooth, can we use \eqref{eq:DF_psh} or \eqref{eq:DF_strict_psh_line_bundle} as a substitute for the usual definition of the Diederich-Forn{\ae}ss index to prove higher-order Sobolev regularity as in \cite{Koh99}, \cite{Har11}, \cite{PiZa14}, \cite{Liu22}, or \cite{LiSt22}?

The plan of the paper is as follows: In Section \ref{sec:Sobolev_spaces}, we outline the key results relating weighted $L^2$ estimates to Sobolev space estimates, and justify applying these results in the setting of arbitrary complex manifolds.  In Section \ref{sec:Hermitian_geometry}, we derive the necessary geometric formulas for our $L^2$ estimates.  In Section \ref{sec:basic_estimate}, we derive a twisted Bochner-Kodaira-Nakano-Morrey-Kohn-H\"ormander identity for domains in Hermitian manifolds.  In Section \ref{sec:solution_operator}, we construct a solution operator to the specific inhomogeneous Cauchy-Riemann equation \eqref{eq:solution_identity} that we will need to estimate the Bergman projection.  In Section \ref{sec:estimates_Bergman}, we use our special solution operator to derive Sobolev Space estimates for the Bergman projection.  In Section \ref{sec:proof_of_C2_theorem}, we complete the proof of Theorem \ref{thm:DF_psh} (as well as Theorems \ref{thm:DF_psh_line_bundle} and \ref{thm:DF_strict_psh_line_bundle}).  Finally, we provide examples in Section \ref{sec:Example}, including the crucial computations on a Hopf manifold.

\section{Sobolev Spaces in Manifolds}

\label{sec:Sobolev_spaces}

Let $M$ be a Hermitian manifold of complex dimension $n$, and let $\Omega\subset M$ be a relatively compact domain with Lipschitz boundary.  Set $\delta(z)=\dist(z,\partial\Omega)$ with respect to the Hermitian metric on $M$. For some $N\geq 1$, let $\{U_j\}_{1\leq j\leq N}$ be an open cover of $\overline\Omega$ such that for each $1\leq j\leq N$ there exists a relatively compact subset $V_j\subset\mathbb{C}^{n}$ and a biholomorphism $\gamma_j:V_j\rightarrow U_j$.  Let $\{\chi_j\}_{1\leq j\leq N}$ be a partition of unity subordinate to $\{U_j\}_{1\leq j\leq N}$ on $\overline{\Omega}$.  For $u\in C^1(\overline\Omega)$ and $s\in\mathbb{R}$, we may define
\[
  \norm{u}_{W^s(\Omega)}=\sum_{j=1}^N\norm{(\chi_j u)\circ\gamma_j}_{W^s(V_j)}.
\]
We set $W^s(\Omega)$ equal to the completion of $C^\infty(\overline\Omega)$ with respect to this norm.  Since $\overline{\Omega}$ is compact and $\partial\Omega$ is Lipschitz, the space $W^s(\Omega)$ is independent of the choice of open cover or partition of unity; see Lemma 1.3.3.1 in \cite{Gri85}.  We define $W^s(\Omega,L)$ similarly, except that we replace $\gamma_j$ with a local trivialization of the holomorphic line bundle $L$.

Fix $0<s<\frac{1}{2}$.  Note that Theorem 1.4.4.3 in \cite{Gri85} implies that
\begin{equation}
\label{eq:weighted_norm_bounded_by_Sobolev_norm}
  \norm{\delta^{-s}u}_{L^2(\Omega)}\leq C\norm{u}_{W^s(\Omega)}
\end{equation}
for some constant $C>0$.  Although Grisvard's result is only stated for Euclidean space, the proof given is local, so it will extend to relatively compact domains in arbitrary manifolds.  It is relevant to note that $\delta$ is locally comparable to the distance function obtained by considering the Euclidean metric on any local coordinate patch.  On the other hand, Proposition 4.15 in \cite{JeKe95} implies that
\begin{equation}
\label{eq:Jerison_Kenig_Proposition}
  \norm{u}_{W^s(\Omega)}\leq C\norm{\delta^{1-s}\abs{d u}+\abs{u}}_{L^2(\Omega)}
\end{equation}
for some constant $C>0$.  Once again, Jerison and Kenig state their result for domains in Euclidean space, but the proof given is local, so it will extend to our case as well.

To complete the proof, we need Lemma 1 in \cite{Det81}, which requires $u$ to be a harmonic function, and hence does not localize as cleanly as the previous results.  Some of these issues can be avoided by replacing ``harmonic" with ``holomorphic," but we must still use some caution.  Let $\{U_j\}_{1\leq j\leq N}$ and $\{\tilde U_j\}_{1\leq j\leq N}$ be two open covers of $\overline\Omega$ such that for every $1\leq j\leq N$ $\overline{\tilde U_j}\subset U_j$ and there exists an open neighborhood $V_j\subset\mathbb{C}^n$ and a biholomorphism $\gamma_j:V_j\rightarrow U_j$.  Fix $r>0$ such that for every $1\leq j\leq N$, if $z\in\gamma_j^{-1}\left[\tilde U_j\right]$ then $B(z,r)\subset V_j$.  Let $u$ be a holomorphic function on $\Omega$, and observe that $u\circ\gamma_j$ is holomorphic on $\gamma_j^{-1}[U_j\cap\Omega]$ for each $1\leq j\leq N$.  Since the Euclidean metric on $T^{1,0}(V_j)$ is comparable to the metric obtained by pulling back the metric on $T^{1,0}(U_j)$ via $\gamma_j$, there exists a constant $\lambda>0$ such that $\gamma_j[B(z,R)]\subset B(\gamma_j(z),\lambda R)$ whenever $z\in\gamma_j^{-1}\left[\tilde U_j\right]$, and $0<R<r$.  We may use the Cauchy estimates (in the form given in $(\star\star)$ in \cite{Det81}) to show that
\begin{equation}
\label{eq:Cauchy_estimate}
  \abs{du(P)}^2\leq\frac{C}{R^{2n+2}}\int_{B(P,\lambda R)}\abs{u(w)}^2 dV_w
\end{equation}
for all $0<R<r$ and $P\in\Omega$ such that $B(P,\lambda R)\subset\Omega$.  Following the proof of Lemma 1 in \cite{Det81}, we may use \eqref{eq:Cauchy_estimate} to show
\[
  \int_{\{z\in\Omega:\delta(z)<2\lambda r\}}\abs{du(z)}^2(\delta(z))^{2-2s}dV_z\leq C\int_{\Omega}\abs{u(z)}^2(\delta(z))^{-2s}dV_z.
\]
Since \eqref{eq:Cauchy_estimate} immediately gives us
\[
  \int_{\{z\in\Omega:\delta(z)\geq 2\lambda r\}}\abs{du(z)}^2(\delta(z))^{2-2s}dV_z\leq C\int_{\Omega}\abs{u(z)}^2(\delta(z))^{-2s}dV_z,
\]
we have
\begin{equation}
\label{eq:Detraz_lemma}
  \norm{\delta^{1-s}du}_{L^2(\Omega)}\leq C\norm{\delta^{-s}u}_{L^2(\Omega)}
\end{equation}
for some $C>0$ whenever $u$ is holomorphic on $\Omega$.  Combining \eqref{eq:Detraz_lemma} with \eqref{eq:Jerison_Kenig_Proposition}, we see that
\begin{equation}
\label{eq:Sobolev_norm_bounded_by_weighted_norm}
  \norm{u}_{W^s(\Omega)}\leq C\norm{\delta^{-s}u}_{L^2(\Omega)}
\end{equation}
for some constant $C>0$ whenever $u$ is holomorphic.

\section{Hermitian Geometry}

\label{sec:Hermitian_geometry}

We begin by summarizing a few well-known facts about the torsion tensor in a Hermitian manifold.
\begin{lem}
\label{lem:torsion_facts}
  Let $M$ be a Hermitian manifold and let $\nabla$ denote the Chern connection on $T^{1,0}(M)$ extended to an affine connection on $T(M)$.  We have
  \begin{equation}
  \label{eq:Kahler_differential}
    d\omega=\tau_\nabla\omega
 \end{equation}
  and
  \begin{equation}
  \label{eq:torsion_vanishing}
    T_\nabla(Z,\bar W)\equiv 0\text{ for all }Z,W\in T^{1,0}(M).
  \end{equation}
\end{lem}

\begin{proof}
  Since $T_\nabla$ is a tensor, it suffices to prove \eqref{eq:torsion_vanishing} for holomorphic vector fields, and then use multi-linearity to generalize to arbitrary sections of $T^{1,0}(M)$.  Choose holomorphic vector fields $Z,W\in T^{1,0}(U)$.  Since $\nabla_{\bar W}Z=0$, $\nabla_Z\bar W=0$, and $[Z,\bar W]=0$, we immediately obtain \eqref{eq:torsion_vanishing}.
  
  Let $U\subset M$ and choose holomorphic vector fields $Z_1,Z_2,W\in T^{1,0}(U)$.  Using the invariant definition of the exterior derivative and $\nabla\omega\equiv 0$, we immediately obtain
  \begin{multline*}
    \partial\omega(Z_1,Z_2,\bar W)=Z_1(\omega(Z_2,\bar W))-Z_2(\omega(Z_1,\bar W))-\omega([Z_1,Z_2],\bar W)=\\
    \omega(\nabla_{Z_1}Z_2,\bar W)-\omega(\nabla_{Z_2}Z_1,\bar W)-\omega([Z_1,Z_2],\bar W)=\omega(T_\nabla(Z_1,Z_2),\bar W).
  \end{multline*}
  Conjugation gives us the analogous result for $\dbar\omega$, and \eqref{eq:Kahler_differential} follows.
\end{proof}

Next, we collect a few formulas to demonstrate how easily the necessary curvatures can be computed in local coordinates:
\begin{lem}
\label{lem:curvature_facts}
  Let $M$ be a Hermitian manifold.  Suppose that on some neighborhood $U\subset M$, $\{W_j\}_{1\leq j\leq n}\subset T^{1,0}(U)$ is a basis of holomorphic vector fields and set $g_{j\bar k}=\left<W_j,W_k\right>$ for all $1\leq j,k\leq n$.  Then
  \begin{equation}
  \label{eq:ricci_computation}
    \Theta_M=-\ddbar\log\det(g_{j\bar k})_{1\leq j,k\leq n}\text{ on }U.
  \end{equation}
  Let $L\rightarrow M$ be a holomorphic line bundle equipped with a $C^2$ Hermitian metric $\left<\cdot,\cdot\right>_L$.  Suppose that $\mathcal{S}$ is a non-vanishing holomorphic section of $L$ on $U$ and set $\psi=-\log\abs{\mathcal{S}}^2_L$.  Then
  \begin{equation}
  \label{eq:line_bundle_curvature_computation}
    \Theta_L=\ddbar\psi\text{ on }U.
  \end{equation}
\end{lem}

\begin{proof}
  As in the proof of Lemma \ref{lem:torsion_facts}, it suffices to consider the actions of $\Theta_M$ and $\Theta_L$ on holomorphic vector fields.

  Let $\{\theta^j\}_{1\leq j\leq n}\subset\Lambda^{1,0}(U)$ be the basis dual to $\{W_j\}_{1\leq j\leq n}$, and set $g^{\bar k j}=\left<\theta^j,\theta^k\right>$.  Let $Z_1,Z_2\in T^{1,0}(U)$ be holomorphic vector fields.  We easily check that
  \[
    \nabla_{Z_1} W_j=\sum_{k,\ell=1}^n (Z_1 g_{j\bar k})g^{\bar k \ell}W_\ell,
  \]
  and so
  \[
    R_\nabla(Z_1,\bar Z_2)W_j=-\nabla_{\bar Z_2}\nabla_{Z_1} W_j=-\sum_{k,\ell=1}^n \bar Z_2((Z_1 g_{j\bar k})g^{\bar k \ell})W_\ell,
  \]
  and hence
  \[
    \Theta_M(Z_1,\bar Z_2)=-\sum_{j,k=1}^n \bar Z_2((Z_1 g_{j\bar k})g^{\bar k j}).
  \]
  By Jacobi's formula for the derivative of the determinant of a matrix of differentiable functions, \eqref{eq:ricci_computation} follows.
  
  Let $\nabla^L$ denote the Chern connection on $L$.  With $Z_1$ and $Z_2$ as before,
  \[
    \nabla^L_{Z_1}\mathcal{S}=e^\psi\left<\nabla^L_{Z_1}\mathcal{S},\mathcal{S}\right>_L\mathcal{S}=e^\psi(Z_1|\mathcal{S}|^2_L)\mathcal{S}=-(Z_1\psi)\mathcal{S},
  \]
  so
  \[
    R_{\nabla^L}(Z_1,\bar Z_2)\mathcal{S}=-\nabla_{\bar Z_2}\nabla_{Z_1}\mathcal{S}=(\bar Z_2 Z_1\psi)\mathcal{S},
  \]
  from which \eqref{eq:line_bundle_curvature_computation} follows.
\end{proof}
Observe that $\Theta_M$ and $\Theta_L$ are locally $\ddbar$-exact, which guarantees that they are globally $\ddbar$-closed, and hence represent Bott-Chern cohomology classes.  More precisely, these are scalar multiples of the first Chern classes of $T^{1,0}(M)$ and $L$, and \eqref{eq:ricci_computation} and \eqref{eq:line_bundle_curvature_computation} are special cases of the same well-known computation of the first Chern class.  As an immediate consequence of Lemma \ref{lem:curvature_facts}, we have
\begin{prop}
\label{prop:Kahler_optimal}
  Let $M$ be a Hermitian manifold of dimension $n\geq 2$ admitting a K\"ahler metric $g$ on $T^{1,0}(M)$ and let $\Omega\subset M$ be a relatively compact domain with Lipschitz boundary.  Let $\tilde g$ denote a Hermitian metric for $T^{1,0}(M)$ (not necessarily K\"ahler) with curvature denoted $\tilde\Theta_M$.  Suppose that for some $0\leq a<b\leq 1$, there exists a Lipschitz defining function $\rho$ for $\Omega$ such that $\rho|_\Omega\in C^2(\Omega)$ and for every $Z\in T^{1,0}(\Omega)$
  \begin{equation}
  \label{eq:DF_psh_Hermitian}
    \left((-\rho)^\eta\tilde\Theta_M-\ddbar(-\rho)^\eta\right)(Z,\bar Z)\geq(-\rho)^\eta\abs{\tau_\nabla Z^\flat}^2
  \end{equation}
  on $\Omega$ for $\eta=a$ and $\eta=b$.  Then there exists $\psi\in C^2(M)$ such that for every $Z\in T^{1,0}(\Omega)$
  \begin{equation}
  \label{eq:DF_psh_Kahler}
    \left((-\rho)^\eta\left(\Theta_M+\ddbar\psi\right)-\ddbar(-\rho)^\eta\right)(Z,\bar Z)\geq 0
  \end{equation}
  on $\Omega$ for $\eta=a$ and $\eta=b$.
\end{prop}

\begin{proof}
  It suffices to note that \eqref{eq:ricci_computation} implies that we locally have
  \[
    \Theta_M-\tilde\Theta_M=-\ddbar\log\det(g_{j\bar k})_{1\leq j,k\leq n}+\ddbar\log\det(\tilde g_{j\bar k})_{1\leq j,k\leq n},
  \]
  so we may locally define
  \[
    \psi=\log\det(g_{j\bar k})_{1\leq j,k\leq n}-\log\det(\tilde g_{j\bar k})_{1\leq j,k\leq n}.
  \]
  Since $\psi$ so-defined is independent of the choice of local coordinates, we obtain a global weight function which allows us to derive \eqref{eq:DF_psh_Kahler} from \eqref{eq:DF_psh_Hermitian}, since the right-hand side of \eqref{eq:DF_psh_Hermitian} is nonnegative.
\end{proof}

Because we will be integrating by parts, we will need an explicit formula for the divergence of a vector field in terms of the Chern connection.  This is easy to compute for K\"ahler manifolds using geodesic coordinates, but for non-K\"ahler manifolds we must compute the impact of the torsion.
\begin{lem}
  Let $M$ be a Hermitian manifold and let $\nabla$ denote the Chern connection on $T^{1,0}(M)$.  Let $U\subset M$ admit a basis $\{W_j\}_{1\leq j\leq n}$ for $T^{1,0}(U)$, and let $\{\theta^j\}_{1\leq j\leq n}\subset\Lambda^{1,0}(U)$ denote the dual basis.  Then for any $C^1$ section $Z\in T^{1,0}(U)$, we have
  \begin{equation}
  \label{eq:divergence_Chern}
    \Div Z=\sum_{j=1}^n\theta^j\left(\nabla_{W_j}Z\right)-\sum_{j=1}^n\theta^j\left(T_\nabla(W_j,Z)\right).
  \end{equation}
\end{lem}

\begin{proof}
  For $U\subset M$, let $\{X_j\}_{1\leq j\leq 2n}\subset T(U)$ be a basis with dual basis $\{\phi^j\}_{1\leq j\leq 2n}\subset\Lambda^1(U)$.  Let $D$ denote the Levi-Civita connection on $T(M)$.  For any $C^1$ section $X\in T(U)$, we have
  \begin{equation}
  \label{eq:divergence_defn}
    \Div X=\sum_{j=1}^{2n}\phi^j\left(D_{X_j}X\right).
  \end{equation}
  This is easily verified at a fixed point $P\in U$ by using geodesic coordinates at $P$ and setting $X_j=\frac{\partial}{\partial x_j}$.  Since \eqref{eq:divergence_defn} is independent of our choice of moving coordinate frame, we see that \eqref{eq:divergence_defn} must hold for any such frame.  In our given bases for $T^{1,0}(U)$ and $\Lambda^{1,0}(U)$, \eqref{eq:divergence_defn} gives us
  \begin{equation}
  \label{eq:divergence_Levi_Civita}
    \Div Z=\sum_{j=1}^n\left(\theta^j\left(D_{W_j}Z\right)+\bar\theta^j\left(D_{\bar W_j}Z\right)\right)
  \end{equation}
  for any $C^1$ section $Z\in T^{1,0}(U)$.
  
  Suppose that $\{W_j\}_{1\leq j\leq n}$ is an orthonormal basis of $C^1$ sections of $T^{1,0}(U)$.  Then since $D$ is compatible with the metric and torsion-free, we have
  \begin{multline*}
    \theta^j\left(D_{W_j}Z\right)=\left<D_{W_j}Z,W_j\right>=\left<[W_j,Z],W_j\right>+\left<D_Z W_j,W_j\right>\\
    =\left<[W_j,Z],W_j\right>-\left<W_j,D_{\bar Z}W_j\right>
    =\left<[W_j,Z],W_j\right>-\left<W_j,[\bar Z,W_j]\right>-\left<W_j,D_{W_j}\bar Z\right>,
  \end{multline*}
  and hence
  \[
    \theta^j\left(D_{W_j}Z\right)+\bar\theta^j\left(D_{\bar W_j}Z\right)=\left<[W_j,Z],W_j\right>-\left<W_j,[\bar Z,W_j]\right>.
  \]
  Expanding the right-hand side using the metric compatibility of $\nabla$, \eqref{eq:torsion_vanishing}, and the fact that $\nabla_{W_j}\bar Z$ is orthogonal to $W_j$, we have
  \begin{align*}
    \theta^j\left(D_{W_j}Z\right)+\bar\theta^j\left(D_{\bar W_j}Z\right)&=\left<\nabla_{W_j}Z,W_j\right>-\left<T_\nabla(W_j,Z),W_j\right>\\
    &=\theta^j\left(\nabla_{W_j}Z\right)-\theta^j\left(T_\nabla(W_j,Z)\right).
  \end{align*}
  We may substitute this in \eqref{eq:divergence_Levi_Civita} to obtain \eqref{eq:divergence_Chern}.  Since \eqref{eq:divergence_Chern} is independent of the choice of coordinates, we may drop the requirement that our moving coordinate frame is orthonormal.

\end{proof}

It will be helpful to a define a modified connection to accommodate the torsion in \eqref{eq:divergence_Chern}.  For $X,Y\in T(M)$ such that $Y$ is a $C^1$ section, we define
\begin{equation}
\label{eq:nabla_T_defn}
  \nabla^T_X Y=\nabla_X Y-T_\nabla(X,Y).
\end{equation}
This also extends to the complexified tangent bundle in the natural way.  Using our new connection defined by \eqref{eq:nabla_T_defn} in \eqref{eq:divergence_Chern}, we have
\begin{equation}
\label{eq:divergence_Chern_T}
  \Div Z=\sum_{j=1}^n\theta^j\left(\nabla^T_{W_j}Z\right).
\end{equation}

We will need the divergence to compute the formal adjoint of a differential operator.  Indeed, if $Z$ is a $C^1$ section of $T^{1,0}(M)$, then the formal adjoint of the differential operator $\bar Z$ with respect to the $L^2$ inner product is given by
\begin{equation}
\label{eq:adjoint_bar_Z}
  \bar Z^*=-Z-\Div Z.
\end{equation}
In practice, we will usually compute the adjoint with respect to the weighted $L^2$ inner product space $L^2(\Omega,\psi)$.  In this case, $\bar Z^*_\psi=e^{\psi}\bar Z^*e^{-\psi}$, so \eqref{eq:adjoint_bar_Z} gives us
\begin{equation}
\label{eq:adjoint_bar_Z_weighted}
  \bar Z^*_\psi=-Z-\Div Z+Z\psi.
\end{equation}

Now, we will compute the commutator of a differential operator with the adjoint of a differential operator.  Since we will only need this formula when both differential operators are holomorphic, we will only carry out the computation in this comparatively simply case.
\begin{lem}
  Let $M$ be a Hermitian manifold, let $\nabla$ denote the Chern connection for $T^{1,0}(M)$ and let $\psi\in C^2(M)$.  For any $U\subset M$ with holomorphic vector fields $Z_1,Z_2\in T^{1,0}(U)$, we have
  \begin{equation}
  \label{eq:commutation_relation_Ricci}
    [\bar Z_2,(\bar Z_1)^*_\psi]=\left(\Theta_M+\ddbar\psi\right)(Z_1,\bar Z_2).
  \end{equation}
\end{lem}

\begin{proof}
  Fix holomorphic vector fields $Z_1,Z_2\in T^{1,0}(M)$.  Since $[Z_1,\bar Z_2]=0$, we may use \eqref{eq:adjoint_bar_Z_weighted} to compute
  \begin{equation}
  \label{eq:commutation_relation}
    [\bar Z_2,(\bar Z_1)^*_\psi]=-\bar Z_2\Div Z_1+\ddbar\psi(Z_1,\bar Z_2).
  \end{equation}
  Let $\{W_j\}_{1\leq j\leq n}$ be a basis of holomorphic vector fields for $T^{1,0}(U)$.  Observe that this also means that the dual basis $\{\theta^j\}_{1\leq j\leq n}$ is a basis of holomorphic $(1,0)$-forms for $\Lambda^{1,0}(U)$.  We use \eqref{eq:divergence_Chern_T} to compute
  \begin{equation}
  \label{eq:divergence_derivative_temp}
    \bar Z_2\Div Z_1=\sum_{j=1}^n\theta^j\left(\nabla_{\bar Z_2}\nabla^T_{W_j}Z_1\right).
  \end{equation}
  Note that \eqref{eq:nabla_T_defn} gives us
  \[
    \nabla_{\bar Z_2}\nabla^T_{W_j}Z_1=\nabla_{\bar Z_2}\nabla_{W_j}Z_1-\nabla_{\bar Z_2}T_\nabla(W_j,Z_1),
  \]
  and since $[W_j,Z_1]$ is also holomorphic, we have
  \[
    \nabla_{\bar Z_2}\nabla^T_{W_j}Z_1=\nabla_{\bar Z_2}\nabla_{Z_1}W_j=-R_\nabla(Z_1,\bar Z_2)W_j.
  \]
  Substituting in \eqref{eq:divergence_derivative_temp}, we have
  \begin{equation}
  \label{eq:divergence_derivative}
    \bar Z_2\Div Z_1=-\Theta_M(Z_1,\bar Z_2)
  \end{equation}
  for any holomorphic vector fields $Z_1,Z_2\in T^{1,0}(U)$.  Substituting \eqref{eq:divergence_derivative} in \eqref{eq:commutation_relation}, we have \eqref{eq:commutation_relation_Ricci}.
\end{proof}

As with any affine connection, $\nabla^T$ extends naturally to a derivation on $\Lambda^q(M)$.  If $\phi\in\Lambda^1(M)$ is a $C^1$ form and $X,Y\in T(M)$, then \eqref{eq:nabla_T_defn} gives us
\[
  (\nabla^T_X \phi)(Y)=X(\phi(Y))-\phi(\nabla^T_X Y)=X(\phi(Y))-\phi(\nabla_X Y)+\phi(T_\nabla(X,Y)),
\]
so
\begin{equation}
\label{eq:nabla_T_forms}
  (\nabla^T_X \phi)(Y)=(\nabla_X\phi)(Y)+\phi(T_\nabla(X,Y)).
\end{equation}

Observe that for any $U\subset M$, if $\{W_j\}_{1\leq j\leq n}$ is a basis for $T^{1,0}(U)$ with dual basis $\{\theta^j\}_{1\leq j\leq n}\subset\Lambda^{1,0}(U)$, then \eqref{eq:nabla_T_forms} gives us
\begin{equation}
\label{eq:nabla_T_difference}
  (\nabla^T_{\bar Z}-\nabla_{\bar Z})u=\sum_{j=1}^n\left((\nabla^T_{\bar Z}-\nabla_{\bar Z})u\right)(\bar W_j)\bar\theta^j=\sum_{j=1}^n u(T_\nabla(\bar Z,\bar W_j))\bar\theta^j
\end{equation}
for any $u\in\Lambda^{0,1}(U)$ and any $Z\in T^{1,0}(U)$.

If we write $g^{\bar j k}=\left<\theta^k,\theta^j\right>$, then may define a family of inner products (with associated norms)
\begin{equation}
\label{eq:nabla_T_inner_product}
  \left<\bar\nabla^1 u,\bar\nabla^2 v\right>=\sum_{j,k=1}^n\left<g^{\bar j k}\nabla^1_{\bar L_j}u,\nabla^2_{\bar L_k}v\right>,
\end{equation}
where $\nabla^1$ and $\nabla^2$ are affine connections (in practice, these will always be linear combinations of $\nabla$ and $\nabla^T$) and $u,v\in\Lambda^{0,1}(U)$ are $C^1$ forms.  We note that the definition given in \eqref{eq:nabla_T_inner_product} is independent of our choice of coordinates.

Finally, we will compute a formula to help evaluate the inner products involving $\tau_\nabla$.
\begin{lem}
  Let $M$ be a Hermitian manifold and let $\nabla$ be the Chern connection on $T^{1,0}(M)$.  For $u,v\in\Lambda^{0,1}(M)$, we have
  \begin{equation}
  \label{eq:tau_inner_product_1}
    \left<\tau_\nabla u,\tau_\nabla v\right>=\frac{1}{2}\left<(\bar\nabla^T-\bar\nabla)u,(\bar\nabla^T-\bar\nabla)v\right>.
  \end{equation}
\end{lem}

\begin{proof}
For $U\subset M$, let $\{W_j\}_{1\leq j\leq n}$ be a basis for $T^{1,0}(U)$ and let $\{\theta^j\}_{1\leq j\leq n}$ be the dual basis for $\Lambda^{1,0}(U)$.  In this moving coordinate frame, we have
\[
  \tau_\nabla u=\sum_{j=1}^{n-1}\sum_{k=j+1}^n u(T_\nabla(\bar W_j,\bar W_k))\bar\theta^j\wedge\bar\theta^k
\]
for any $u\in\Lambda^{0,1}(U)$.  Since \eqref{eq:tau_inner_product_1} is independent of the choice of moving coordinate frame, we may assume that $\{W_j\}_{1\leq j\leq n}$ is an orthonormal basis for $T^{1,0}(U)$.  In this case, for any $u,v\in\Lambda^{0,1}(U)$ we have
\[
  \left<\tau_\nabla u,\tau_\nabla v\right>=\sum_{j=1}^{n-1}\sum_{k=j+1}^n u(T_\nabla(\bar W_j,\bar W_k))\overline{v(T_\nabla(\bar W_j,\bar W_k))}.
\]
Using the skew-symmetry of $T_\nabla$, we may rewrite this in the form
\begin{equation}
\label{eq:tau_inner_product_local}
  \left<\tau_\nabla u,\tau_\nabla v\right>=\frac{1}{2}\sum_{j,k=1}^n u(T_\nabla(\bar W_j,\bar W_k))\overline{v(T_\nabla(\bar W_j,\bar W_k))}.
\end{equation}
Combining \eqref{eq:tau_inner_product_local} with \eqref{eq:nabla_T_difference} and \eqref{eq:nabla_T_inner_product}, we have \eqref{eq:tau_inner_product_1}.
\end{proof}

\section{The Basic Estimate}

\label{sec:basic_estimate}

Before stating the basic estimate, we will need some useful notation.  Let $U\subset M$ admit a basis $\{W_j\}_{1\leq j\leq n}$ for $T^{1,0}(U)$ and a dual basis $\{\theta^j\}_{1\leq j\leq n}$ for $\Lambda^{1,0}(U)$.  For any $\Theta\in\Lambda^{1,1}(U)$ satisfying $\overline{\Theta(Z,\bar W)}=\Theta(W,\bar Z)$ for all $Z,W\in T^{1,0}(U)$, we define a self-adjoint endomorphism on $\Lambda^{0,1}(U)$ by 
\begin{equation}
\label{eq:Theta_operator_defn}
  i\Theta u=\sum_{j,k=1}^n(\Theta(W_k,\bar W_j))\left<u,\bar\theta^k\right>\bar\theta^j.
\end{equation}

The following identity is well-known on K\"ahler manifolds: see Proposition 2.4 in \cite{Str10} for twisted estimates in $\mathbb{C}^n$ and Theorem 3.1 in \cite{McVa15} for twisted estimates in K\"ahler manifolds (with the correction that $\varphi$ should be $\psi$ in (12)).  Both of these references provide extensive history and applications for twisted estimates in the $L^2$-theory for $\dbar$.  The non-K\"ahler case has been studied by Griffiths ((7.14) in \cite{Gri66}) and Demailly ((3.1) in \cite{Dem86} or VII-(2.1) in \cite{Dem12}), but these authors do not use a twist factor.

\begin{prop}[Twisted Bochner-Kodaira-Morrey-Kohn-H\"ormander]
\label{prop:twisted_BKMKH}
  Let $\Omega\subset M$ be a relatively compact domain with $C^3$ boundary, and let $\rho\in C^3(M)$ be a defining function for $\Omega$.  Let $\psi,\kappa\in C^2(\overline\Omega)$ satisfy $\inf_\Omega\kappa>0$.  Then for all $u\in C^2_{0,1}(\overline\Omega)\cap\dom\dbar^*$, we have
  \begin{multline}
  \label{eq:Q_identity}
    \norm{\sqrt{\kappa}\dbar u}^2_{L^2(\Omega,\psi)}+\norm{\sqrt{\kappa}\dbar^*_\psi u}^2_{L^2(\Omega,\psi)}=\norm{\sqrt{\kappa}\bar\nabla^T u}^2_{L^2(\Omega,\psi)}-\norm{\sqrt{\kappa}\tau_\nabla u}^2_{L^2(\Omega,\psi)}\\
    +2\re\left(\dbar^*_\psi u,\left<u,\dbar\kappa\right>\right)_{L^2(\Omega,\psi)}
    +\left(\left(i\kappa\Theta_M+i\kappa\ddbar\psi-i\ddbar\kappa\right)u,u\right)_{L^2(\Omega,\psi)}\\
    +\left(|\nabla\rho|^{-1}\kappa(i\ddbar\rho)u,u\right)_{L^2(\partial\Omega,\psi)}.
  \end{multline}
\end{prop}

\begin{proof}
Let $U\subset M$ be an open set on which there exists a basis of holomorphic sections $\{W_j\}_{1\leq j\leq n}\subset T^{1,0}(U)$ with dual basis $\{\theta^j\}_{1\leq j\leq n}\subset\Lambda^{1,0}(U)$.  Assume that $U\cap\Omega\neq\emptyset$.  For $\theta\in C^1_{1,0}(\overline\Omega)$ supported in $U\cap\overline\Omega$, we may use the invariant definition of exterior derivative to check that
\begin{align*}
  \dbar\bar\theta&=\sum_{j=1}^{n-1}\sum_{k=j+1}^n\dbar\bar\theta(\bar W_j,\bar W_k)\bar\theta^j\wedge\bar\theta^k\\
  &=\sum_{j=1}^{n-1}\sum_{k=j+1}^n(\bar W_j(\bar\theta(\bar W_k))-\bar W_k(\bar\theta(\bar W_j))-\bar\theta([\bar W_j,\bar W_k]))\bar\theta^j\wedge\bar\theta^k,
\end{align*}
and so \eqref{eq:nabla_T_forms} gives us
\begin{align*}
  \dbar\bar\theta&=\sum_{j=1}^{n-1}\sum_{k=j+1}^n((\nabla_{\bar W_j}\bar\theta)(\bar W_k)-(\nabla_{\bar W_k}\bar\theta)(\bar W_j)+\bar\theta(T_\nabla(\bar W_j,\bar W_k)))\bar\theta^j\wedge\bar\theta^k\\
  &=\frac{1}{2}\sum_{j=1}^n\bar\theta^j\wedge\left(\nabla_{\bar W_j}+\nabla_{\bar W_j}^T\right)\bar\theta.
\end{align*}
Since $\nabla$ and $\nabla^T$ are both derivations, for any $1\leq q\leq n$ we have
\begin{equation}
\label{eq:dbar_formula}
  \dbar u=\frac{1}{2}\sum_{j=1}^n\bar\theta^j\wedge\left(\nabla_{\bar W_j}+\nabla_{\bar W_j}^T\right)u,
\end{equation}
where $u\in C^1_{0,q}(\overline\Omega)$ is supported in $U\cap\overline{\Omega}$.

Now we fix $u,v\in C^2_{0,1}(\overline\Omega)\cap\dom\dbar^*$ supported in $U\cap\overline{\Omega}$.  We define the Dirichlet form
\begin{equation}
\label{eq:Dirichlet_form}
  Q_\psi(u,v)=\left(\kappa\dbar u,\dbar v\right)_{L^2(\Omega,\psi)}+\left(\kappa\dbar^*_\psi u,\dbar^*_\psi v\right)_{L^2(\Omega,\psi)}.
\end{equation}
Since we have a moving coordinate frame on $U$, we write $u^j=\left<u,\bar\theta^j\right>$ and $v^j=\left<v,\bar\theta^j\right>$ for all $1\leq j\leq n$.  With this notation, we may compute
\begin{equation}
\label{eq:dbar_adjoint_formula}
  \dbar^*_\psi u=\sum_{j=1}^n(\bar W_j)^*_\psi u^j.
\end{equation}

To evaluate the second term in \eqref{eq:Dirichlet_form}, we use \eqref{eq:dbar_adjoint_formula} to obtain
\[
  \left(\kappa\dbar^*_\psi u,\dbar^*_\psi v\right)_{L^2(\Omega,\psi)}=\sum_{j,k=1}^n\left(\kappa(\bar W_k)^*_\psi u^k,(\bar W_j)^*_\psi v^j\right)_{L^2(\Omega,\psi)}.
\]
Since $v\in\dom\dbar^*$, $\left<v,\dbar\rho\right>\equiv 0$ on $\partial\Omega$, so we may integrate by parts with no boundary term and obtain
\[
  \left(\kappa\dbar^*_\psi u,\dbar^*_\psi v\right)_{L^2(\Omega,\psi)}=\left(\dbar^*_\psi u,\left<v,\dbar\kappa\right>\right)_{L^2(\Omega,\psi)}
  +\sum_{j,k=1}^n\left(\kappa\bar W_j(\bar W_k)^*_\psi u^k,v^j\right)_{L^2(\Omega,\psi)}.
\]
We may use \eqref{eq:commutation_relation_Ricci} to commute the differential operators and, using the notation introduced in \eqref{eq:Theta_operator_defn}, we have
\begin{align*}
  \left(\kappa\dbar^*_\psi u,\dbar^*_\psi v\right)_{L^2(\Omega,\psi)}&=\left(\dbar^*_\psi u,\left<v,\dbar\kappa\right>\right)_{L^2(\Omega,\psi)}+\left(\kappa\left(i\Theta_M+i\ddbar\psi\right)u,v\right)_{L^2(\Omega,\psi)}\\
  &\qquad+\sum_{j,k=1}^n\left(\kappa(\bar W_k)^*_\psi\bar W_j u^k,v^j\right)_{L^2(\Omega,\psi)}.
\end{align*}
This time, if we integrate by parts with the remaining adjoint differential operator, we will need to introduce a boundary term:
\begin{multline*}
  \left(\kappa\dbar^*_\psi u,\dbar^*_\psi v\right)_{L^2(\Omega,\psi)}=\left(\dbar^*_\psi u,\left<v,\dbar\kappa\right>\right)_{L^2(\Omega,\psi)}+\left(\kappa\left(i\Theta_M+i\ddbar\psi\right)u,v\right)_{L^2(\Omega,\psi)}\\
  +\sum_{j,k=1}^n\left((W_k\kappa)\bar W_j u^k,v^j\right)_{L^2(\Omega,\psi)}+\sum_{j,k=1}^n\left(\kappa\bar W_j u^k,\bar W_k v^j\right)_{L^2(\Omega,\psi)}\\
  -\sum_{j,k=1}^n\left(|\nabla\rho|^{-1}(W_k\rho)\kappa\bar W_j u^k,v^j\right)_{L^2(\partial\Omega,\psi)}.
\end{multline*}
Since $u,v\in\dom\dbar^*$, the standard Morrey trick for rewriting the boundary integral (see (4.3.8)-(4.3.9) in \cite{ChSh01} or the discussion of (2.32) in \cite{Str10}) and the notation from \eqref{eq:Theta_operator_defn} give us
\begin{multline*}
  \left(\kappa\dbar^*_\psi u,\dbar^*_\psi v\right)_{L^2(\Omega,\psi)}=\left(\dbar^*_\psi u,\left<v,\dbar\kappa\right>\right)_{L^2(\Omega,\psi)}+\left(\kappa\left(i\Theta_M+i\ddbar\psi\right)u,v\right)_{L^2(\Omega,\psi)}\\
  +\sum_{j,k=1}^n\left((W_k\kappa)\bar W_j u^k,v^j\right)_{L^2(\Omega,\psi)}+\sum_{j,k=1}^n\left(\kappa\bar W_j u^k,\bar W_k v^j\right)_{L^2(\Omega,\psi)}\\
  +\left(|\nabla\rho|^{-1}\kappa(i\ddbar\rho)u,v\right)_{L^2(\partial\Omega,\psi)}.
\end{multline*}
Finally, we integrate the $\bar W_j$ by parts in the first term on the second line and use \eqref{eq:Theta_operator_defn} and \eqref{eq:dbar_adjoint_formula} to simplify the resulting terms, so that we have
\begin{multline}
\label{eq:dbar_adjoint_inner_product}
  \left(\kappa\dbar^*_\psi u,\dbar^*_\psi v\right)_{L^2(\Omega,\psi)}=\left(\dbar^*_\psi u,\left<v,\dbar\kappa\right>\right)_{L^2(\Omega,\psi)}+\left(\left<u,\dbar\kappa\right>,\dbar^*_\psi v\right)_{L^2(\Omega,\psi)}\\
  +\left(\left(i\kappa\Theta_M+i\kappa\ddbar\psi-i\ddbar\kappa\right)u,v\right)_{L^2(\Omega,\psi)}
  +\sum_{j,k=1}^n\left(\kappa\bar W_j u^k,\bar W_k v^j\right)_{L^2(\Omega,\psi)}\\
  +\left(|\nabla\rho|^{-1}\kappa(i\ddbar\rho)u,v\right)_{L^2(\partial\Omega,\psi)}.
\end{multline}

To evaluate the first term in \eqref{eq:Dirichlet_form}, we will not need to integrate by parts, so it will suffice to consider pointwise computations.  We use \eqref{eq:dbar_formula} to obtain
\[
  \left<\dbar u,\dbar v\right>=\frac{1}{4}\sum_{j,k=1}^n\left<\bar\theta^j\wedge\left(\nabla_{\bar W_j}+\nabla^T_{\bar W_j}\right)u,\bar\theta^k\wedge\left(\nabla_{\bar W_k}+\nabla^T_{\bar W_k}\right)v\right>.
\]
We may use the definition of the inner product for $2$-forms and the notation introduced in \eqref{eq:nabla_T_inner_product} to compute
\begin{multline}
\label{eq:dbar_inner_product_1}
  \left<\dbar u,\dbar v\right>=\frac{1}{4}\left<(\bar\nabla+\bar\nabla^T)u,(\bar\nabla+\bar\nabla^T)v\right>\\
  -\frac{1}{4}\sum_{j,k=1}^n\left<\left<\left(\nabla_{\bar W_j}+\nabla^T_{\bar W_j}\right)u,\bar\theta^k\right>,\left<\left(\nabla_{\bar W_k}+\nabla^T_{\bar W_k}\right)v,\bar\theta^j\right>\right>.
\end{multline}
To evaluate the first term on the right-hand side of \eqref{eq:dbar_inner_product_1}, we use \eqref{eq:tau_inner_product_1} to write
\begin{multline}
\label{eq:dbar_inner_product_1a}
  \frac{1}{4}\left<(\bar\nabla+\bar\nabla^T)u,(\bar\nabla+\bar\nabla^T)v\right>=\left<\bar\nabla^T u,\bar\nabla^T v\right>\\
  -\frac{1}{2}\sum_{j=1}^n\left<\bar\nabla u,(\bar\nabla^T-\bar\nabla)v\right>-\frac{1}{2}\left<(\bar\nabla^T-\bar\nabla)u,\bar\nabla v\right>
  -\frac{3}{2}\left<\tau_\nabla u,\tau_\nabla v\right>.
\end{multline}
To evaluate the second term on the right-hand side of \eqref{eq:dbar_inner_product_1}, we observe that since $W_k$ is holomorphic, $\theta^k$ must also be holomorphic for all $1\leq k\leq n$, so we have
\[
  \frac{1}{2}\left<\left(\nabla_{\bar W_j}+\nabla^T_{\bar W_j}\right)u,\bar\theta^k\right>=\bar W_j u^k+\frac{1}{2}\left<\left(\nabla^T_{\bar W_j}-\nabla_{\bar W_j}\right)u,\bar\theta^k\right>,
\]
and then \eqref{eq:nabla_T_difference} gives us
\[
  \frac{1}{2}\left<\left(\nabla_{\bar W_j}+\nabla^T_{\bar W_j}\right)u,\bar\theta^k\right>=\bar W_j u^k+\frac{1}{2}\sum_{\ell=1}^n u(T_\nabla(\bar W_j,\bar W_\ell))g^{\bar\ell k}.
\]
Expanding the corresponding inner product, we have
\begin{multline}
\label{eq:dbar_inner_product_1b_temp}
  \frac{1}{4}\sum_{j,k=1}^n\left<\left<\left(\nabla_{\bar W_j}+\nabla^T_{\bar W_j}\right)u,\bar\theta^k\right>,\left<\left(\nabla_{\bar W_k}+\nabla^T_{\bar W_k}\right)v,\bar\theta^j\right>\right>=\\
  \sum_{j,k=1}^n\left<\bar W_j u^k,\bar W_k v^j\right>
  +\frac{1}{2}\sum_{j,k,\ell=1}^n\left<\bar W_j u^k,v(T_\nabla(\bar W_k,\bar W_\ell))g^{\bar\ell j}\right>\\
  +\frac{1}{2}\sum_{j,k,\ell=1}^n\left<u(T_\nabla(\bar W_j,\bar W_\ell))g^{\bar\ell k},\bar W_k v^j\right>
  -\frac{1}{2}\left<\tau_\nabla u,\tau_\nabla v\right>.
\end{multline}
Compare \eqref{eq:tau_inner_product_local} for the justification of the coefficient of $\frac{1}{2}$ in the final term, and note that we have used the skew-symmetry of $T_\nabla$ to compute the sign of the final term.  By \eqref{eq:nabla_T_difference} and \eqref{eq:nabla_T_inner_product} we have
\[
  \left<\bar\nabla u,(\bar\nabla^T-\bar\nabla)v\right>=\sum_{j,k,\ell=1}^n g^{\bar j\ell}\left<\bar W_j u^k,v(T_\nabla(\bar W_\ell,\bar W_k))\right>,
\]
so we may use the skew-symmetry of $T_\nabla$ to substitute this in \eqref{eq:dbar_inner_product_1b_temp} and obtain
\begin{multline}
\label{eq:dbar_inner_product_1b}
  \frac{1}{4}\sum_{j,k=1}^n\left<\left<\left(\nabla_{\bar W_j}+\nabla^T_{\bar W_j}\right)u,\bar\theta^k\right>,\left<\left(\nabla_{\bar W_k}+\nabla^T_{\bar W_k}\right)v,\bar\theta^j\right>\right>=\\
  \sum_{j,k=1}^n\left<\bar W_j u^k,\bar W_k v^j\right>
  -\frac{1}{2}\sum_{j,k=1}^n\left<\bar\nabla u,(\bar\nabla^T-\bar\nabla)v\right>
  -\frac{1}{2}\sum_{j,k=1}^n\left<(\bar\nabla^T-\bar\nabla)u,\bar\nabla v\right>\\
  -\frac{1}{2}\left<\tau_\nabla u,\tau_\nabla v\right>.
\end{multline}
Now, we may substitute \eqref{eq:dbar_inner_product_1a} and \eqref{eq:dbar_inner_product_1b} in \eqref{eq:dbar_inner_product_1} to obtain
\begin{equation}
\label{eq:dbar_inner_product}
  \left<\dbar u,\dbar v\right>=\left<\bar\nabla^T u,\bar\nabla^T v\right>
  -\left<\tau_\nabla u,\tau_\nabla v\right>
  -\sum_{j,k=1}^n\left<\bar W_j u^k,\bar W_k v^j\right>.
\end{equation}

Finally, we substitute \eqref{eq:dbar_inner_product} and \eqref{eq:dbar_adjoint_inner_product} into \eqref{eq:Dirichlet_form} to obtain
\begin{multline}
\label{eq:Q_inner_product}
  Q_\psi(u,v)=\left(\kappa\bar\nabla^T u,\bar\nabla^T v\right)_{L^2(\Omega,\psi)}-\left(\kappa\tau_\nabla u,\tau_\nabla v\right)_{L^2(\Omega,\psi)}\\
  +\left(\dbar^*_\psi u,\left<v,\dbar\kappa\right>\right)_{L^2(\Omega,\psi)}
  +\left(\left<u,\dbar\kappa\right>,\dbar^*_\psi v\right)_{L^2(\Omega,\psi)}\\
  +\left(\left(i\kappa\Theta_M+i\kappa\ddbar\psi-i\ddbar\kappa\right)u,v\right)_{L^2(\Omega,\psi)}
  +\left(|\nabla\rho|^{-1}\kappa(i\ddbar\rho)u,v\right)_{L^2(\partial\Omega,\psi)}.
\end{multline}
Since every term in \eqref{eq:Q_inner_product} is independent of the choice of local coordinates, we may use a partition of unity to prove that \eqref{eq:Q_inner_product} holds for all $u,v\in C^2_{0,1}(\overline\Omega)\cap\dom\dbar^*$.  Now, \eqref{eq:Q_identity} follows from \eqref{eq:Q_inner_product} by setting $v=u$.
\end{proof}

\begin{cor}
  Let $\Omega\subset M$ be a relatively compact pseudoconvex domain with $C^3$ boundary.  Let $\psi,\kappa\in C^2(\overline\Omega)$ satisfy $\inf_\Omega\kappa>0$.  Then for all $u\in L^2_{0,1}(\Omega,\psi)\cap\dom\dbar\cap\dom\dbar^*_\psi$, we have
  \begin{multline}
  \label{eq:Q_estimate}
    \norm{\sqrt{\kappa}\dbar u}^2_{L^2(\Omega,\psi)}+\norm{\sqrt{\kappa}\dbar^*_\psi u}^2_{L^2(\Omega,\psi)}\geq-\norm{\sqrt{\kappa}\tau_\nabla u}^2_{L^2(\Omega,\psi)}\\
    +2\re\left(\dbar^*_\psi u,\left<u,\dbar\kappa\right>\right)_{L^2(\Omega,\psi)}
    +\left(\left(i\kappa\Theta_M+i\kappa\ddbar\psi-i\ddbar\kappa\right)u,u\right)_{L^2(\Omega,\psi)}.
  \end{multline}
\end{cor}

\begin{proof}
  When $u\in C^2_{0,1}(\overline\Omega)\cap\dom\dbar^*$, \eqref{eq:Q_estimate} follows immediately from \eqref{eq:Q_identity} because the $\bar\nabla^T$ term is non-negative and the boundary term is also non-negative when $\Omega$ is pseudoconvex.  When $\partial\Omega$ is $C^3$, we may use a standard density result, such as Lemma 4.3.2 in \cite{ChSh01} or Proposition 2.3 in \cite{Str10}.  Although these results are stated for $\mathbb{C}^n$, the proofs are local and easily adaptable to complex manifolds.
\end{proof}

\section{A Solution Operator}

\label{sec:solution_operator}

It has long been known that estimates for the Bergman projection could be obtained by first solving the inhomogeneous Cauchy-Riemann equations with $L^2$ estimates.  Indeed, Kohn's formula $P=I-S\dbar$, where $S$ is the $L^2$ minimal solution operator for the inhomogeneous Cauchy-Riemann equations, provides an explicit relation between the two operators.  Our key observation is that it suffices to solve \eqref{eq:solution_identity}, which is an important special case of the inhomogeneous Cauchy-Riemann equations.
\begin{prop}
\label{prop:solution_operator}
  Let $\Omega\subset M$ be a relatively compact domain with $C^3$ boundary.  Let $\psi,\kappa\in C^2(\overline{\Omega})$ satisfy
  \begin{enumerate}
    \item $\inf_{\Omega}\kappa>0$ and
    \item there exists a constant $B>0$ such that
      \begin{equation}
      \label{eq:curvature_hypothesis}
        \left<\left(i\kappa\Theta_M+i\kappa\ddbar\psi-i\ddbar\kappa\right)\bar\theta,\bar\theta\right>-\abs{\sqrt{\kappa}\tau_\nabla \bar\theta}^2\geq B\kappa^{-1}\abs{\left<\bar\theta,\dbar\kappa\right>}^2
      \end{equation}
      on $\Omega$ for all $\theta\in\Lambda^{1,0}(\Omega)$.
  \end{enumerate}
  Then for every $f\in\ker\dbar\cap L^2(\Omega,\psi)$, there exists a unique $u\in L^2(\Omega,\psi)$ such that
  \begin{enumerate}
    \item $u$ is orthogonal to $L^2(\Omega,\psi)\cap\ker\dbar$ with respect to the $L^2(\Omega,\psi)$ inner product,
    \item
    \begin{equation}
    \label{eq:solution_estimate}
      \norm{u}_{L^2(\Omega,\psi-\log\kappa)}\leq\frac{1}{\sqrt{1+B}}\norm{f}_{L^2(\Omega,\psi-\log\kappa)},
    \end{equation}
    and
    \item
    \begin{equation}
    \label{eq:solution_identity}
      \dbar u=\kappa^{-1}(f-u)\dbar\kappa
    \end{equation}
    on $\Omega$ in the distribution sense.
  \end{enumerate}
\end{prop}

\begin{rem}
\label{rem:comparable_norms}
  In Proposition \ref{prop:solution_operator}, $\log\kappa$ is uniformly bounded on $\Omega$, so we have
  \[
    L^2(\Omega,\psi-\log\kappa)=L^2(\Omega,\psi+\log\kappa)=L^2(\Omega,\psi)
  \]
  with comparable norms.
\end{rem}

\begin{proof}
  Fix $v_1\in L^2_{0,1}(\Omega,\psi)\cap\ker\dbar\cap\dom\dbar^*_\psi$.  Substituting \eqref{eq:curvature_hypothesis} in \eqref{eq:Q_estimate}, we have
  \[
    \norm{\sqrt{\kappa}\dbar^*_\psi v_1}^2_{L^2(\Omega,\psi)}\geq
    2\re\left(\dbar^*_\psi v_1,\left<v_1,\dbar\kappa\right>\right)_{L^2(\Omega,\psi)}+B\norm{\kappa^{-\frac{1}{2}}\left<v_1,\dbar\kappa\right>}^2_{L^2(\Omega,\psi)},
  \]
  which gives us
  \[
    (1+B^{-1})\norm{\dbar^*_\psi v_1}^2_{L^2(\Omega,\psi-\log\kappa)}\geq
    B^{-1}\norm{\dbar^*_\psi v_1+B\kappa^{-1}\left<v_1,\dbar\kappa\right>}^2_{L^2(\Omega,\psi-\log\kappa)}.
  \]
  Since the weighted Bergman projection $P_{\psi-\log\kappa}$ is trivially bounded in $L^2(\Omega,\psi-\log\kappa)$, with norm one, we have
  \[
    (1+B)\norm{\dbar^*_\psi v_1}^2_{L^2(\Omega,\psi-\log\kappa)}\geq
    \norm{P_{\psi-\log\kappa}\left(\dbar^*_\psi v_1+B\kappa^{-1}\left<v_1,\dbar\kappa\right>\right)}^2_{L^2(\Omega,\psi-\log\kappa)}.
  \]
  By \eqref{eq:dbar_adjoint_formula} and \eqref{eq:adjoint_bar_Z_weighted}, 
  \[
    \dbar^*_\psi v_1-\dbar^*_{\psi-\log\kappa} v_1=\sum_{j=1}^n(L_j\log\kappa)\left<v_1,\bar\theta^j\right>=\kappa^{-1}\left<v_1,\dbar\kappa\right>,
  \]
  and so we have
  \begin{multline*}
    (1+B)\norm{\dbar^*_\psi v_1}^2_{L^2(\Omega,\psi-\log\kappa)}\geq\\
    \norm{P_{\psi-\log\kappa}\left(\dbar^*_{\psi-\log\kappa} v_1+(1+B)\kappa^{-1}\left<v_1,\dbar\kappa\right>\right)}^2_{L^2(\Omega,\psi-\log\kappa)}.
  \end{multline*}
  However, $\dbar^*_{\psi-\log\kappa}$ maps into the orthogonal complement of $\ker\dbar$ with respect to the inner product on $L^2(\Omega,\psi-\log\kappa)$, and hence it maps into the kernel of $P_{\psi-\log\kappa}$.  Thus we have
  \begin{equation}
  \label{eq:basic_estimate}
    \norm{\dbar^*_\psi v_1}^2_{L^2(\Omega,\psi-\log\kappa)}\geq
    (1+B)\norm{P_{\psi-\log\kappa}\left(\kappa^{-1}\left<v_1,\dbar\kappa\right>\right)}^2_{L^2(\Omega,\psi-\log\kappa)}.
  \end{equation}
  
  Let $H_{\psi}$ denote the closure of the range of $\dbar^*_\psi$ in $L^2(\Omega,\psi)$.  As noted in Remark \ref{rem:comparable_norms}, this is also the closure of the range of $\dbar^*_\psi$ in $L^2(\Omega,\psi-\log\kappa)$. Using the usual Hodge theory arguments, $H_{\psi}$ is also the orthogonal complement of $L^2(\Omega,\psi)\cap\ker\dbar$ with respect to the inner product on $L^2(\Omega,\psi)$.

  Fix $f\in\ker\dbar\cap L^2(\Omega,\psi)$.  For $v\in L^2_{0,1}(\Omega,\psi)\cap\dom\dbar\cap\dom\dbar^*_\psi$, we may write $v=v_1+v_2$, where $v_1\in\ker\dbar$ and $v_2$ is orthogonal to $\ker\dbar$ with respect to $L^2_{0,1}(\Omega,\psi)$.  Then since $f\dbar\kappa\in\ker\dbar$, we have
  \[
    \left(f\dbar\kappa,v\right)_{L^2(\Omega,\psi)}=\left(f\dbar\kappa,v_1\right)_{L^2(\Omega,\psi)}.
  \]
  Now (noting Remark \ref{rem:comparable_norms} again) $f=P_{\psi-\log\kappa}f$, so we have
  \begin{multline*}
    \abs{\left(f\dbar\kappa,v\right)_{L^2(\Omega,\psi)}}=\abs{\left(f,\left<v_1,\dbar\kappa\right>\right)_{L^2(\Omega,\psi)}}=\abs{\left(f,\kappa^{-1}\left<v_1,\dbar\kappa\right>\right)_{L^2(\Omega,\psi-\log\kappa)}}\\
    \abs{\left(f,P_{\psi-\log\kappa}\left(\kappa^{-1}\left<v_1,\dbar\kappa\right>\right)\right)_{L^2(\Omega,\psi-\log\kappa)}}.
  \end{multline*}
  By \eqref{eq:basic_estimate}, we have
  \[
    \abs{\left(f\dbar\kappa,v\right)_{L^2(\Omega,\psi)}}\leq\frac{1}{\sqrt{1+B}}\norm{f}_{L^2(\Omega,\psi-\log\kappa)}\norm{\dbar^*_\psi v_1}_{L^2(\Omega,\psi-\log\kappa)},
  \]
  but $\ker\dbar^*_\psi$ contains the orthogonal complement of $\ker\dbar$ with respect to $L^2(\Omega,\psi)$, so
  \begin{equation}
  \label{eq:bounded_functional}
    \abs{\left(f\dbar\kappa,v\right)_{L^2(\Omega,\psi)}}\leq\frac{1}{\sqrt{1+B}}\norm{f}_{L^2(\Omega,\psi-\log\kappa)}\norm{\dbar^*_\psi v}_{L^2(\Omega,\psi-\log\kappa)}.
  \end{equation}
  Now \eqref{eq:bounded_functional} and Remark \ref{rem:comparable_norms} imply that the map $\dbar^*_\psi v\mapsto(v,f\dbar\kappa)_{L^2(\Omega,\psi)}$ induces a well-defined bounded linear functional on $F:H_{\psi}\rightarrow\mathbb{C}$ satisfying the operator norm bound
  \[
    \norm{F}=\sup_{h\in H_{\psi}\backslash\{0\}}\frac{\abs{F h}}{\norm{h}_{L^2(\Omega,\psi-\log\kappa)}}\leq\frac{1}{\sqrt{1+B}}\norm{f}_{L^2(\Omega,\psi-\log\kappa)}.
  \]
  By the Riesz Representation Theorem, there exists a unique $u\in H_{\psi}$ satisfying \eqref{eq:solution_estimate}
  and
  \[
    \left(u,\dbar^*_\psi v\right)_{L^2(\Omega,\psi-\log\kappa)}=\left(f\dbar\kappa,v\right)_{L^2(\Omega,\psi)}
  \]
  for all $v\in L^2_{0,1}(\Omega,\psi)\cap\dom\dbar\cap\dom\dbar^*_\psi$.  Then $\dbar(\kappa u)=f\dbar\kappa$ in the distribution sense.  Expanding this derivative, we see that $u$ satisfies \eqref{eq:solution_identity}.
  
\end{proof}

\section{Estimates for the Bergman Projection}

\label{sec:estimates_Bergman}

\begin{prop}
\label{prop:Bergman_projection}
  Let $\Omega\subset M$ satisfy $\Omega=\bigcup_{j\in\mathbb{N}}\Omega_j$, where $\{\Omega_j\}_{j\in\mathbb{N}}$ is an increasing sequence of relatively compact pseudoconvex domains with $C^3$ boundaries satisfying
  \begin{equation}
  \label{eq:vanishing_measure}
    \lim_{j\rightarrow\infty}\abs{\Omega\backslash\Omega_j}=0.
  \end{equation}
  Let $\psi,\kappa\in C(\Omega)$ satisfy
  \begin{enumerate}
    \item $\kappa\in L^\infty(\Omega)$ and
    \item $\kappa>0$ on $\Omega$.
  \end{enumerate}
  For each $j\in\mathbb{N}$, let $\psi_j,\kappa_j\in C^2(\overline{\Omega_j})$ satisfy
  \begin{enumerate}
    \item there exists a constant $B>0$ independent of $j$ such that \eqref{eq:curvature_hypothesis} holds for each $\kappa_j$ and $\psi_j$ on $\Omega_j$ for all $v\in\Lambda^{0,1}(\Omega_j)$, and
    \item for every $\epsilon>0$ there exists $k\in\mathbb{N}$ such that for all $j\geq k$ we have
    \begin{equation}
    \label{eq:weight_approximation}
      \psi\leq\psi_j\leq\psi+\epsilon
    \end{equation}
    and
    \begin{equation}
    \label{eq:twist_approximation}
      e^{-\epsilon}\kappa\leq\kappa_j\leq\kappa
    \end{equation}
    on $\Omega_j$.
  \end{enumerate}
  Then for every $v\in L^2(\Omega,\psi)$,
  \begin{equation}
  \label{eq:Bergman_projection_weighted_estimate}
    \norm{P_\psi v}_{L^2(\Omega,\psi-\log\kappa)}\leq
    \frac{\sqrt{1+B}}{\sqrt{1+B}-1}\norm{v}_{L^2(\Omega,\psi-\log\kappa)}.
  \end{equation}
  Hence, $P_\psi$ admits a unique continuous extension to $L^2(\Omega,\psi-\log\kappa)$ satisfying \eqref{eq:Bergman_projection_weighted_estimate}.
\end{prop}

\begin{proof}
  The following technique leading to \eqref{eq:Bergman_projection_Boas_Straube_formula} was introduced by Boas and Straube in \cite{BoSt90}.  Let $v\in L^2(\Omega,\psi)\cap L^2(\Omega,\psi-\log\kappa)$ and $h\in L^2(\Omega,\psi)\cap\ker\dbar$.  Then
  \[
    (P_\psi v,h)_{L^2(\Omega,\psi)}=(v,h)_{L^2(\Omega,\psi)}=(\kappa v,h)_{L^2(\Omega,\psi+\log\kappa)}
  \]
  Now
  \[
    \norm{\kappa v}_{L^2(\Omega,\psi+\log\kappa)}=\norm{v}_{L^2(\Omega,\psi-\log\kappa)}<\infty,
  \]
  so $P_{\psi+\log\kappa}(\kappa v)$ is well defined, and we have
  \begin{multline*}
    (P_\psi v,h)_{L^2(\Omega,\psi)}=(P_{\psi+\log\kappa}(\kappa v),h)_{L^2(\Omega,\psi+\log\kappa)}
    =(\kappa^{-1}P_{\psi+\log\kappa}(\kappa v),h)_{L^2(\Omega,\psi)}=\\
    (P_\psi(\kappa^{-1}P_{\psi+\log\kappa}(\kappa v)),h)_{L^2(\Omega,\psi)}.
  \end{multline*}
  Since this holds for all $h\in L^2(\Omega,\psi)\cap\ker\dbar$,
  \begin{equation}
  \label{eq:Bergman_projection_Boas_Straube_formula}
    P_\psi v=P_\psi(\kappa^{-1}P_{\psi+\log\kappa}(\kappa v))\text{ for all }v\in L^2(\Omega,\psi)\cap L^2(\Omega,\psi-\log\kappa).
  \end{equation}
  
  Observe that in this context Remark \ref{rem:comparable_norms} applies on $\Omega_j$.  Since $\kappa\in L^\infty(\Omega)$, we also have
  \[
    L^2(\Omega,\psi+\log\kappa)\subset L^2(\Omega,\psi)\subset L^2(\Omega,\psi-\log\kappa)
  \]
  
  Fix $v\in L^2(\Omega,\psi)$.  On $\Omega_j$, set
  \[
    u_j=(P_{\psi_j}-I)\left(\kappa_j^{-1}P_{\psi_j+\log\kappa_j}\left(\kappa_j v|_{\Omega_j}\right)\right)
  \]
  and $f_j=P_{\psi_j}\left(v|_{\Omega_j}\right)$.  Then
  \[
    \dbar u_j=\kappa_j^{-2}\dbar\kappa_j\wedge P_{\psi_j+\log\kappa_j}\left(\kappa_j v|_{\Omega_j}\right)=\kappa_j^{-1}(f_j-u_j)\dbar\kappa_j,
  \]
  so $u_j$ satisfies \eqref{eq:solution_identity} on $\Omega_j$ by \eqref{eq:Bergman_projection_Boas_Straube_formula}.  Since $u_j$ is orthogonal to $\ker\dbar$ with respect to $L^2(\Omega_j,\psi_j)$, $u_j$ must be the unique solution constructed in Proposition \ref{prop:solution_operator}, and hence $u_j$ also satisfies \eqref{eq:solution_estimate}. By \eqref{eq:Bergman_projection_Boas_Straube_formula},
  \[
    f_j-u_j=\kappa_j^{-1}P_{\psi_j+\log\kappa_j}\left(\kappa_j v|_{\Omega_j}\right),
  \]
  and so
  \begin{multline*}
    \norm{f_j-u_j}_{L^2(\Omega_j,\psi_j-\log\kappa_j)}=\norm{P_{\psi_j+\log\kappa_j}\left(\kappa_j v|_{\Omega_j}\right)}_{L^2(\Omega_j,\psi_j+\log\kappa_j)}\leq\\
    \norm{v|_{\Omega_j}}_{L^2(\Omega_j,\psi_j-\log\kappa_j)}.
  \end{multline*}
  Combined with \eqref{eq:solution_estimate}, we have
  \begin{multline*}
    \norm{f_j}_{L^2(\Omega_j,\psi_j-\log\kappa_j)}\leq\norm{f_j-u_j}_{L^2(\Omega_j,\psi_j-\log\kappa_j)}+\norm{u_j}_{L^2(\Omega_j,\psi_j-\log\kappa_j)}\leq\\
    \norm{v|_{\Omega_j}}_{L^2(\Omega_j,\psi_j-\log\kappa_j)}+\frac{1}{\sqrt{1+B}}\norm{f_j}_{L^2(\Omega_j,\psi_j-\log\kappa_j)},
  \end{multline*}
  or equivalently,
  \begin{equation}
  \label{eq:u_approximation_estimate}
    \norm{f_j}_{L^2(\Omega_j,\psi_j-\log\kappa_j)}\leq
    \frac{\sqrt{1+B}}{\sqrt{1+B}-1}\norm{v|_{\Omega_j}}_{L^2(\Omega_j,\psi_j-\log\kappa_j)}.
  \end{equation}
  
  For $\epsilon>0$, choose $k\in\mathbb{N}$ such that for all $j\geq k$ \eqref{eq:weight_approximation} holds on $\Omega_j$, \eqref{eq:twist_approximation} holds on $\Omega_j$,
  \[
    \norm{v}_{L^2(\Omega\backslash\Omega_j,\psi)}\leq\epsilon,
  \]
  and
  \[
    \norm{v}_{L^2(\Omega\backslash\Omega_j,\psi-\log\kappa)}\leq\epsilon.
  \]
  Here, we have used \eqref{eq:vanishing_measure}, the fact that $|v|^2 e^{-\psi}$ is integrable, and $\kappa\in L^\infty(\Omega)$.  Define $\tilde f_j\in L^2(\Omega,\psi-\log\kappa)$ by $\tilde f_j=f_j$ on $\Omega_j$ and $\tilde f_j=v$ on $\Omega\backslash\Omega_j$.  Then \eqref{eq:u_approximation_estimate} gives us
  \[
    e^{-\epsilon}\norm{\tilde f_j}_{L^2(\Omega,\psi)}\leq\norm{v}_{L^2(\Omega_j,\psi_j)}+e^{-\epsilon}\norm{v}_{L^2(\Omega\backslash\Omega_j,\psi)}
    \leq\norm{v}_{L^2(\Omega,\psi)}+e^{-\epsilon}\epsilon,
  \]
  so
  \[
    \norm{\tilde f_j}_{L^2(\Omega,\psi)}\leq e^{\epsilon}\norm{v}_{L^2(\Omega,\psi)}+\epsilon
  \]
  for all $j\geq k$.  By restricting to a subsequence, we may assume that $\set{\tilde f_j}_{j\in\mathbb{N}}$ converges weakly with respect to $L^2(\Omega,\psi)$ to some element $f\in L^2(\Omega,\psi)$ satisfying 
  \[
    \norm{f}_{L^2(\Omega,\psi)}\leq e^{\epsilon}\norm{v}_{L^2(\Omega,\psi)}+\epsilon.
  \]
  Since $f$ is independent of $\epsilon$, we must have
  \[
    \norm{f}_{L^2(\Omega,\psi)}\leq \norm{v}_{L^2(\Omega,\psi)}.
  \]
  
  By similar reasoning, \eqref{eq:u_approximation_estimate}, \eqref{eq:weight_approximation}, and \eqref{eq:twist_approximation} give us
  \begin{multline*}
    e^{-2\epsilon}\norm{\tilde f_j}_{L^2(\Omega,\psi-\log\kappa)}\leq
    \frac{\sqrt{1+B}}{\sqrt{1+B}-1}\norm{v}_{L^2(\Omega_j,\psi_j-\log\kappa_j)}+e^{-2\epsilon}\norm{v}_{L^2(\Omega\backslash\Omega_j,\psi-\log\kappa)}\\
    \leq\frac{\sqrt{1+B}}{\sqrt{1+B}-1}\norm{v}_{L^2(\Omega,\psi-\log\kappa)}+e^{-2\epsilon}\epsilon,
  \end{multline*}
  or equivalently
  \[
    \norm{\tilde f_j}_{L^2(\Omega,\psi-\log\kappa)}\leq
    e^{2\epsilon}\frac{\sqrt{1+B}}{\sqrt{1+B}-1}\norm{v}_{L^2(\Omega,\psi-\log\kappa)}+\epsilon
  \]
  for all $j\geq k$.  Now,
  \begin{multline*}
    \norm{f}_{L^2(\Omega,\psi-\log\kappa)}=\sup_{h\in L^2(\Omega,\psi+\log\kappa)\backslash\{0\}}\frac{\abs{(f,h)_{L^2(\Omega,\psi)}}}{\norm{h}_{L^2(\Omega,\psi+\log\kappa)}}\\
    =\sup_{h\in L^2(\Omega,\psi+\log\kappa)\backslash\{0\}}\lim_{j\rightarrow\infty}\frac{\abs{(\tilde f_j,h)_{L^2(\Omega,\psi)}}}{\norm{h}_{L^2(\Omega,\psi+\log\kappa)}}\leq\lim_{j\rightarrow\infty}\norm{\tilde f_j}_{L^2(\Omega,\psi-\log\kappa)}\\
    \leq e^{2\epsilon}\frac{\sqrt{1+B}}{\sqrt{1+B}-1}\norm{v}_{L^2(\Omega,\psi-\log\kappa)}+\epsilon.
  \end{multline*}
  Once again, $f$ is independent of $\epsilon$, so we must have
  \begin{equation}
  \label{eq:f_estimate}
    \norm{f}_{L^2(\Omega,\psi-\log\kappa)}\leq
    \frac{\sqrt{1+B}}{\sqrt{1+B}-1}\norm{v}_{L^2(\Omega,\psi-\log\kappa)}.
  \end{equation}
  
  Fix $h\in L^2(\Omega,\psi)\cap\ker\dbar$.  Then
  \[
    (f-v,h)_{L^2(\Omega,\psi)}=\lim_{j\rightarrow\infty}(\tilde f_j-v,h)_{L^2(\Omega,\psi)}=
    \lim_{j\rightarrow\infty}(f_j-v,h)_{L^2(\Omega_j,\psi)}.
  \]
  On $\Omega_j$, $v-f_j=(I-P_{\psi_j})v$, which is orthogonal to $\ker\dbar$ with respect to $L^2(\Omega_j,\psi_j)$, and hence we may write
  \[
    (f-v,h)_{L^2(\Omega,\psi)}=\lim_{j\rightarrow\infty}\left((f_j-v,h)_{L^2(\Omega_j,\psi)}-(f_j-v,h)_{L^2(\Omega_j,\psi_j)}\right).
  \]
  By \eqref{eq:weight_approximation}, whenever $j\geq k$, we have
  \[
    0\leq e^{-\psi}-e^{-\psi_j}=e^{-\psi}(1-e^{\psi-\psi_j})\leq e^{-\psi}(1-e^{-\epsilon}),
  \]
  and so
  \begin{multline*}
    \abs{(f_j-v,h)_{L^2(\Omega_j,\psi)}-(f_j-v,h)_{L^2(\Omega_j,\psi_j)}}\leq\\
    (1-e^{-\epsilon})\norm{f_j-v}_{L^2(\Omega_j,\psi)}\norm{h}_{L^2(\Omega_j,\psi)}.
  \end{multline*}
  We conclude that $(f-v,h)_{L^2(\Omega,\psi)}=0$, so $f-v$ is orthogonal to $L^2(\Omega,\psi)\cap\ker\dbar$ with respect to the $L^2(\Omega,\psi)$ inner product.
  
  Finally, we let $\theta\in C^\infty_{1,0}(\Omega)$ be compactly supported, so we may assume that $k\in\mathbb{N}$ is sufficiently large so that the support of $\theta$ is contained in $\Omega_j$ for all $j\geq k$.  Then
  \begin{multline*}  
    (f,\dbar^*_{\psi}\bar\theta)_{L^2(\Omega,\psi)}=\lim_{j\rightarrow\infty}(\tilde f_j,\dbar^*_{\psi}\bar\theta)_{L^2(\Omega,\psi)}\\
    =\lim_{j\rightarrow\infty}(f_j,\dbar^*_{\psi}\bar\theta)_{L^2(\Omega_j,\psi)}+\lim_{j\rightarrow\infty}(v,\dbar^*_{\psi}\bar\theta)_{L^2(\Omega\backslash\Omega_j,\psi)}.
  \end{multline*}
  The first limit vanishes because $f_j\in\ker\dbar$ and $\bar\theta\in\dom\dbar^*_{\psi}$ on $\Omega_j$ for $j\geq k$.  The second limit vanishes because of \eqref{eq:vanishing_measure} and the fact that $\left<v,\dbar^*_{\psi}\bar\theta\right>e^{-\psi}$ is integrable.  Hence, $f\in\ker\dbar$ in the sense of distributions, but by Weyl's lemma we may conclude that $f$ is holomorphic.
  
  Now $f$ is holomorphic and $f-v$ is orthogonal to $L^2(\Omega,\psi)\cap\ker\dbar$ with respect to the $L^2(\Omega,\psi)$ inner product, but these properties uniquely characterize $P_\psi v$, so we must have $f=P_\psi v$, and hence \eqref{eq:Bergman_projection_weighted_estimate} follows from \eqref{eq:f_estimate}.
\end{proof}

\section{Proof of Theorem \ref{thm:DF_psh}}

\label{sec:proof_of_C2_theorem}

\begin{lem}
\label{lem:DF_psh}
  Let $M$ be a Hermitian manifold of dimension $n\geq 2$ and let $\Omega\subset M$ be a relatively compact domain with Lipschitz boundary.  Let $\psi\in C^2(M)$.  Suppose that for some $0<\eta<1$ and $B>0$, there exists a Lipschitz defining function $\rho$ for $\Omega$ such that $\rho|_\Omega\in C^2(\Omega)$ and
  \begin{multline}
  \label{eq:DF_psh_B}
    \left((-\rho)^\eta\left(\Theta_M+\ddbar\psi\right)-\ddbar(-\rho)^\eta\right)(Z,\bar Z)-(-\rho)^\eta\abs{\tau_\nabla Z^\flat}^2\geq\\
     B\eta^2(-\rho)^{\eta-2}\abs{\partial\rho(Z)}^2
  \end{multline}
  holds on $\Omega$ for all $Z\in T^{1,0}(\Omega)$.  Then the Bergman projection $P$ is continuous in $W^{\eta/2}(\Omega)$.
\end{lem}

\begin{proof}
If $\kappa=(-\rho)^\eta$, then
\begin{equation}
\label{eq:partial_sigma}
  \partial\kappa=-\eta(-\rho)^{\eta-1}\partial\rho.
\end{equation}
Using \eqref{eq:partial_sigma}, we see that \eqref{eq:DF_psh_B} and \eqref{eq:curvature_hypothesis} are equivalent to each other.  Using Proposition \ref{prop:Bergman_projection} with $\Omega_j=\Omega$, $\psi_j=\psi|_{\Omega_j}$, and $\kappa_j=\kappa|_{\Omega_j}$, we have
\[
  \norm{(-\rho)^{\eta/2}P_\psi v}_{L^2(\Omega,\psi)}\leq\frac{\sqrt{1+B}}{\sqrt{1+B}-1}\norm{(-\rho)^{\eta/2}v}_{L^2(\Omega,\psi)}.
\]
Since the Bergman projection is self-adjoint, we immediately obtain the dual estimate
\[
  \norm{(-\rho)^{-\eta/2}P_\psi v}_{L^2(\Omega,\psi)}\leq\frac{\sqrt{1+B}}{\sqrt{1+B}-1}\norm{(-\rho)^{-\eta/2}v}_{L^2(\Omega,\psi)}.
\]
Since $P_\psi v$ is necessarily a holomorphic function and $\psi$ is uniformly bounded on $\overline\Omega$, we may use \eqref{eq:weighted_norm_bounded_by_Sobolev_norm} and \eqref{eq:Sobolev_norm_bounded_by_weighted_norm} to prove that these weighted estimates imply Sobolev space estimates for the Bergman projection.

\end{proof}

To prove Theorem \ref{thm:DF_psh}, we first fix $\frac{a}{2}<s<\frac{b}{2}$ and $Z\in T^{1,0}(\Omega)$.  For $\psi\in C^2(M)$, $\rho\in C^2(\Omega)$, and $0<\eta<1$, we define a $(1,1)$-form $\Psi_{\psi,\rho,\eta}$ by
\begin{equation}
\label{eq:Psi_defn}
  \Psi_{\psi,\rho,\eta}(Z,\bar W)=\left(\Theta_M+\ddbar\psi+\eta(-\rho)^{-1}\ddbar\rho\right)(Z,\bar W)-\left<\tau_\nabla Z^\flat,\tau_\nabla W^\flat\right>
\end{equation}
for all $Z,W\in T^{1,0}(\Omega)$.  We compute
\begin{equation}
\label{eq:Psi_interpolation}
  \Psi_{\psi,\rho,2s}(Z,\bar Z)=\frac{b-2s}{b-a}\Psi_{\psi,\rho,a}(Z,\bar Z)+\frac{2s-a}{b-a}\Psi_{\psi,\rho,b}(Z,\bar Z).
\end{equation}

When $\kappa=(-\rho)^\eta$, we may differentiate \eqref{eq:partial_sigma} again to obtain
\begin{equation}
\label{eq:ddbar_sigma}
  \ddbar\kappa=-\eta(-\rho)^{\eta-1}\ddbar\rho-\eta(1-\eta)(-\rho)^{\eta-2}\partial\rho\wedge\dbar\rho.
\end{equation}
Since \eqref{eq:DF_psh} holds for $\eta=b$, \eqref{eq:ddbar_sigma} gives us
\begin{equation}
\label{eq:Psi_a}
  \Psi_{\psi,\rho,b}(Z,\bar Z)+b(1-b)(-\rho)^{-2}\abs{\partial\rho(Z)}^2\geq 0,
\end{equation}
and similarly we have
\begin{equation}
\label{eq:Psi_b}
  \Psi_{\psi,\rho,a}(Z,\bar Z)+a(1-a)(-\rho)^{-2}\abs{\partial\rho(Z)}^2\geq 0.
\end{equation}
Substituting \eqref{eq:Psi_a} and \eqref{eq:Psi_b} into \eqref{eq:Psi_interpolation} gives us
\[
  \Psi_{\psi,\rho,2s}(Z,\bar Z)\geq-\left(ab+2s-2s(a+b)\right)(-\rho)^{-2}\abs{\partial\rho(Z)}^2,
\]
and hence
\[
  \Psi_{\psi,\rho,2s}(Z,\bar Z)+2s(1-2s)(-\rho)^{-2}\abs{\partial\rho(Z)}^2\geq(b-2s)(2s-a)(-\rho)^{-2}\abs{\partial\rho(Z)}^2.
\]
Hence we have \eqref{eq:DF_psh_B} with $B=(2s)^{-2}(b-2s)(2s-a)$, so the conclusion follows from Lemma \ref{lem:DF_psh}.

When $0<s\leq\frac{a}{2}$, we may use interpolation to complete the proof of the result.

\section{Examples}
\label{sec:Example}

\subsection{Hopf Manifolds}
\label{sec:hopf_manifolds}

Fix $a\in\mathbb{C}$ such that $0<|a|<1$.  For $n\geq 1$, let $\mathbb{H}^n=(\mathbb{C}^n\backslash\{0\})/\sim$, where $z\sim w$ if $z=a^j w$ for some $j\in\mathbb{Z}$.  Then $\mathbb{H}^n$ is a Hopf manifold, which is known to not admit a global K\"ahler metric when $n\geq 2$.  Endow $T^{1,0}(M)$ with a metric induced by the K\"ahler form $\omega=\frac{i}{2|z|^2}\sum_{j=1}^n dz_j\wedge d\bar z_j$.  Observe that $W_1=\sum_{j=1}^n z_j\frac{\partial}{\partial z_j}$ induces a global unit-length holomorphic vector field on $M$.  By \eqref{eq:ricci_computation}, we have
\[
  \Theta_M=\ddbar\log|z|^{2n}=\frac{n}{|z|^2}\ddbar|z|^2-\frac{n}{|z|^4}\partial|z|^2\wedge\dbar|z|^2,
\]
and hence
\begin{equation}
\label{eq:hopf_curvature}
  \Theta_M(Z,\bar Z)=n|Z|^2-n\abs{\left<Z,W_1\right>}^2\text{ for all }Z\in T^{1,0}(M).
\end{equation}
Therefore, $i\Theta_M$ has $n-1$ eigenvalues equal to $n$ and one eigenvalue equal to zero with eigenvector $W_1$.  Since
\[
  d\omega=-d(\log|z|^2)\wedge\omega,
\]
\eqref{eq:Kahler_differential} gives us
\[
  \tau_\nabla\theta=-\partial(\log|z|^2)\wedge\theta
\]
for all $\theta\in\Lambda^{1,0}(M)$.  On $U\subset M$, let $\{W_j\}_{1\leq j\leq n}\in T^{1,0}(U)$ be an orthonormal basis with $W_1$ as above.  Note that $W_1\log|z|^2\equiv 1$ and $W_j\log|z|^2\equiv 0$ for all $2\leq j\leq n$.  For $Z\in T^{1,0}(U)$, we have
\[
  (\tau_\nabla Z^\flat)(W_j,W_k)=0\text{ if }2\leq j<k\leq n
\]
and
\[
  (\tau_\nabla Z^\flat)(W_1,W_j)=-\left<W_j,Z\right>\text{ if }2\leq j\leq n.
\]
Hence,
\[
  \abs{\tau_\nabla Z^\flat}^2=\sum_{j=2}^n\abs{\left<W_j,Z\right>}^2,
\]
or equivalently
\begin{equation}
\label{eq:hopf_torsion}
  \abs{\tau_\nabla Z^\flat}^2=\abs{Z}^2-\abs{\left<Z,W_1\right>}^2.
\end{equation}
Substituting \eqref{eq:hopf_curvature} and \eqref{eq:hopf_torsion} in \eqref{eq:DF_psh}, we see that \eqref{eq:DF_psh} holds on a domain $\Omega\subset M$ if and only if
\begin{equation}
\label{eq:hopf_DF_PSH}
  (n-1)(-\rho)^\eta\left(\abs{Z}^2-\abs{\left<Z,W_1\right>}^2\right)-\ddbar(-\rho)^\eta(Z,\bar Z)\geq 0
\end{equation}
on $\Omega$ for all $Z\in T^{1,0}(\Omega)$.  This means that on $\mathbb{H}^n$ with $n\geq 2$, the relevant curvature dominates the torsion, and hence Theorem \ref{thm:DF_psh} will apply on any domain admitting a positive Diederich-Forn{\ae}ss index.  Moreover, \eqref{eq:hopf_DF_PSH} gains additional positivity for $Z$ orthogonal to $W_1$, so Theorem \ref{thm:DF_psh} may apply with a weaker notion of the Diederich-Forn{\ae}ss index.



\subsection{An Application of Theorem \ref{thm:DF_psh}}
\label{sec:C2_example}

Define $D$ to be the set of all $z\in\mathbb{C}$ such that $\abs{\im z}<1$ and $\abs{\re z}<1$.  A defining function for $D$ is given by
\[
  r(z)=\begin{cases}
    \frac{((\im z)^2-1)((\re z)^2-1)}{|z|^2-2}&z\in D\\
    \max\{(\im z)^2-1,(\re z)^2-1\}&z\notin D
  \end{cases}.
\]
Then $r$ is a Lipschitz defining function for $D$ that satisfies $r|_D\in C^2(D)$.  Furthermore, $r$ is strictly convex on $D$.  

For $n\geq 2$, set $M=\mathbb{C}\times\mathbb{H}^{n-1}$, where $\mathbb{H}^{n-1}$ is defined in Section \ref{sec:hopf_manifolds}.  We will equip $\mathbb{C}$ with the Euclidean metric, $\mathbb{H}^{n-1}$ with the metric given in Section \ref{sec:hopf_manifolds}, and $M$ with the induced product metric.  Let $\Omega=D\times\mathbb{H}^{n-1}$ with defining function $\rho(z)=r(z_1)$.  Then $\Omega$ is a pseudoconvex domain with Levi-flat boundary, but $\Omega$ is not Stein because it contains compact Riemann surfaces in its interior.  Moreover, when $n\geq 3$, there does not exist a K\"ahler metric on any neighborhood of $\overline{\Omega}$.  We set $\psi\equiv 0$, so that \eqref{eq:DF_psh} holds trivially when $\eta=0$.  Since $\rho$ is plurisubharmonic, \eqref{eq:DF_psh} holds for all $0\leq\eta\leq 1$, and hence Theorem \ref{thm:DF_psh} implies that the Bergman projection is continuous in $W^s(\Omega)$ for all $0\leq s<\frac{1}{2}$.

Note that any holomorphic function on $\Omega$ is necessarily constant on each copy of $\mathbb{H}^{n-1}$ in $\Omega$, so the Bergman space of $\Omega$ is naturally isomorphic to the Bergman space for $D$.  This suggests that we can also prove regularity of the Bergman projection on $\Omega$ using known properties of the Bergman projection on $D$.  In fact, the Bergman kernel can be computed explicitly using a Schwarz–Christoffel mapping from the unit disc onto $D$.

Large classes of domains with Levi-flat boundaries have been constructed using holomorphic bundles over a complex manifold (see \cite{DiOh07} or \cite{DeFo20} and the references therein).  Some of these domains are Stein (see \cite{Ohs82}, for example).  Others admit defining function $\rho$ for which $-(-\rho)^\eta$ is plurisubharmonic for some $0<\eta<1$, but not strictly plurisubharmonic (see \cite{FuSh18}, for example).  In the example considered in \cite{FuSh18}, the natural metric is locally a product of Poincar\'{e} metrics, which has negative definite Ricci curvature on $M$, and hence \eqref{eq:DF_psh} cannot hold on $\Omega$ for any $0<\eta<1$.  See \cite{Ada21} for a detailed analysis of weighted Bergman spaces on such domains.  It would be of great interest to find a domain with Levi-flat boundary that is not a product domain on which the hypotheses of Theorems \ref{thm:DF_psh} or \ref{thm:DF_psh_line_bundle} hold.
\bibliographystyle{amsplain}
\bibliography{harrington}
\end{document}